\documentclass[a4paper]{amsproc}
\usepackage{amssymb}
\usepackage{amscd}
\usepackage{pstricks}
\usepackage{tikz}
\usepackage[all,cmtip]{xy}
\usepackage{pst-3dplot}
\usepackage{pst-all}
\usepackage{pstricks-add}
\usepackage{enumitem}
\usepackage{pst-solides3d}
\usepackage{color, graphics}

\theoremstyle{plain}
 \newtheorem{thm}{Theorem}[section]

 \newtheorem{prop}{Proposition}[section]
 
 \newtheorem{cor}{Corollary}[section]
\theoremstyle{definition}
 \newtheorem{exm}{Example}[section]
 \newtheorem{remark}{Remark}[section]
\newtheorem{dfn}{Definition}[section]

\numberwithin{equation}{section}

\DeclareMathOperator{\diag}{diag}

\DeclareMathOperator{\grad}{grad}

\title[Regressions and Pencils of Confocal Quadrics]{Orthogonal and Linear Regressions and Pencils of Confocal Quadrics}

\author[V. Dragovi\'c, B. Gaji\'c]{\bfseries Vladimir Dragovi\'c and Borislav Gaji\'c}

\address{
Department of Mathematical Sciences  \\
The University of Texas at Dallas   \\
Richardson, TX\\
USA\\
Mathematical Institute\\
of the Serbian Academy of Sciences and Arts\\
Belgrade\\
Serbia}
\email{Vladimir.Dragovic@utdallas.edu}

\address{
Mathematical Institute\\
of the Serbian Academy of Sciences and Arts\\
Belgrade\\
Serbia}
\email{gajab@mi.sanu.ac.rs}

\subjclass[2010]{}


\begin{document}

\maketitle

\begin{abstract}
This paper enhances and develops bridges between statistics, mechanics, and geometry.
For a given system of points in $\mathbb R^k$ representing a sample of full rank, we construct an explicit pencil of confocal quadrics with the following properties:
(i) All the hyperplanes for which the hyperplanar moments of inertia for the given system of points are equal, are tangent to  the same quadrics
from the pencil of quadrics.  As an application, we develop regularization procedures for the orthogonal least square method, analogues of lasso and ridge methods from linear regression. (ii) For any given point $P$ among all the hyperplanes that contain it, the best fit is the tangent hyperplane to the quadric from the confocal pencil corresponding to the maximal Jacobi coordinate of the point $P$; the worst fit among the hyperplanes containing $P$ is the tangent hyperplane to the ellipsoid from the confocal pencil that contains $P$. The confocal pencil of quadrics provides a universal tool to solve the restricted principal component analysis restricted at any given point. Both results (i) and (ii) can be seen as generalizations of the classical result of Pearson on orthogonal regression. They have natural and important applications in the statistics of the  errors-in-variables models (EIV). For the classical linear regressions we provide a geometric characterization of hyperplanes of least squares in a given direction among all hyperplanes which contain a given point.  The obtained results have applications in restricted regressions, both ordinary and orthogonal ones. For the latter, a new formula for test statistic is derived. The developed methods and results are illustrated in natural statistics examples.
\end{abstract}

\section{Introduction}\label{sec:intro}
The aim of this paper is to further develop and enhance bridges between three disciplines: statistics, mechanics, and geometry.
More precisely, we will explore and employ links between quadrics, moments of inertia, and  regressions, in both the ordinary linear and the orthogonal settings. As it is well known and as it will be discussed later in detail, orthogonal regression plays a crucial role in Errors-in-Variables models (EIV), as seen in \cite{CR0,CR,Full}. The EIV models have various applications in different areas in science, as seen in \cite{AB,AFB,BGSY,CMM,Ke,MOW}.

This cross-fertilization between statistics, mechanics, and geometry appears to be beneficial for each of them. While individual quadrics have been in use in statistics, as reviewed in Section \ref{sec:bestfit} (see also \cite{FMF, Ro}), the novelty of this paper lies in the construction of a new object, a confocal pencil of quadrics, associated to a given data set, which we attest to be a powerful and universal tool to study the data.
More about confocal pencils of quadrics and associated Jacobi elliptic coordinates in the space of an arbitrary dimension $k$ is given in  Supplementary material, see \cite{Jac} and also e.g. \cite{ArnoldMMM}, \cite{DR}.

For the reader's sake, we will briefly present the basics of the pencil of conics and associated Jacobi elliptic coordinates in the plane. This corresponds to the case $k=2$ and should provide enough intuition for the case of an arbitrary number of dimensions.
Let an ellipse $\mathcal{E}$  be given in the plane by the equation:
$$
\mathcal{E}: \frac{x^2}{\alpha}+\frac{y^2}{\beta}=1,\quad \alpha>\beta>0.
$$
All the conics in the plane, sharing the same focal points with $\mathcal{E}$, form a pencil of confocal conics, ${\mathcal C}_{\lambda}$, defined by the equation:
\begin{equation}\label{eq:confocal2d}
{\mathcal C}_{\lambda}\ :\ \frac {{x}^2}{\alpha-\lambda}+ \frac{{y}^2}{\beta-\lambda }=1,
\end{equation}
see Figure \ref{fig:confocal}, where $F_1, F_2$ are the common focal points.

The parameter $\lambda$ is real, and each fixed $\lambda$ defines a conic ${\mathcal C}_{\lambda}$.
The family ${\mathcal C}_{\lambda}$ \eqref{eq:confocal2d} contains two types of nondegenerate conics: ellipses when $\lambda<\beta$, and hyperbolas when $\lambda\in(\beta,\alpha)$, see Figure \ref{fig:confocal}. There are also two degenerate conics in the confocal pencil: the $x$-axis for $\lambda=\beta$, and the $y$-axis for $\lambda=\alpha$.

Each point in the plane, $P(x_0,y_0)$, which is not a focus of the confocal pencil of conics (i.e. $P\notin \{F_1, F_2\}$), lies at exactly two conics $\mathcal C_{\lambda_1}$ and $\mathcal C_{\lambda_2}$ from \eqref{eq:confocal2d} -- one ellipse and one hyperbola, which are orthogonal to each other at the intersection point $P$,  see Figure \ref{fig:confocal}.

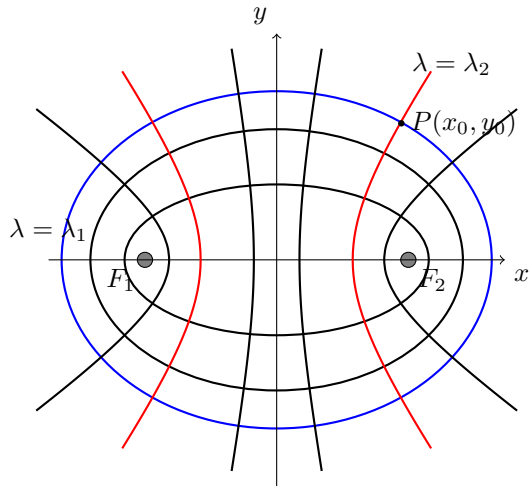
\begin{figure}[h] \centering
\begin{tikzpicture}

\draw[blue, thick](0,0) circle [x radius={sqrt(8)}, y radius={sqrt(5)}];
\draw[thick](0,0) circle [x radius=2, y radius=1];
\draw[thick](0,0) circle [x radius={sqrt(6)}, y radius={sqrt(3)}];

\draw[black,fill=gray]({sqrt(3)},0) circle [x radius=0.1,y radius=0.1] node[anchor=north west] {$F_2$};
\draw[black,fill=gray]({-sqrt(3)},0) circle [x radius=0.1,y radius=0.1] node[anchor=north east] {$F_1$};

\draw[domain=-2:2,smooth,thick,variable=\t] plot ({sqrt(2*(\t*\t+1))},{\t});
\draw[domain=-2:2,smooth,thick,variable=\t] plot ({-sqrt(2*(\t*\t+1))},{\t});

\draw[->] (-3,0) -- (3,0) node[anchor=north west] {$x$ };
\draw[->] (0,-3) -- (0,3) node[anchor=south east] {$y$};

\draw[red, domain=-2.5:2.5,smooth,thick,variable=\t] plot ({sqrt(1+\t*\t/2))},{\t});
\draw[red, domain=-2.5:2.5,smooth,thick,variable=\t] plot ({-sqrt(1+\t*\t/2))},{\t});

\draw[domain=-2.8:2.8,smooth,thick,variable=\t] plot ({0.3*sqrt(1+\t*\t/2.7))},{\t});
\draw[domain=-2.8:2.8,smooth,thick,variable=\t] plot ({-0.3*sqrt(1+\t*\t/2.7))},{\t});

\filldraw[black] (1.64,1.81) circle (1pt) node[anchor=west]{$P(x_0,y_0)$};

\node at (2.3,2.6){$\lambda=\lambda_2$};
\node at  (-3,0.4) {$\lambda=\lambda_1$};

\end{tikzpicture}
\caption{A confocal pencil of conics in plane; the Jacobi coorinates}
\label{fig:confocal}
\end{figure}

For a given point $P(x_0,y_0)$, the parameters $\lambda_1$ and $\lambda_2$ of the conics from the confocal pencil that contain $P$ are the solutions of the quadratic equation \eqref{eq:confocal2d} in $\lambda$, where the given Cartesian coordinates $x=x_0$ and $y=y_0$ of the point $P$ are known parameters. By definition, these $(\lambda_1, \lambda_2)$ are the Jacobi elliptic coordinates of the point $P$ associated with the given confocal pencil of conics.

 Jacobi (1804-1851) introduced Jacobi's elliptic coordinates in arbitrary dimension in the 26-th Lecture of his Course on Dynamics, delivered in K\"oningsberg in 1842/43.
The Jacobi Lectures \cite{Jac} were published posthumously by Clebsch in 1865 in German. Jacobi used these coordinates in Lecture 26 to solve one of the important open problems of the first half of XIX century-- to integrate the equations of geodesics  on an ellipsoid in $\mathbb R^k$. Jacobi opened the Lecture 26 with the famous lines:

\emph{``The main difficulty in integrating a given differential equation lies in introducing convenient variables, which there is no rule for finding. Therefore, we must travel the reverse path and after finding some notable substitution, look for problems to which it can be successfully applied."
(see also e.g. \cite{ArnoldMMM}, section 47, \cite{Jur}).}

Following this ideology of Jacobi, in the present paper, we search for problems in Statistics to which the Jacobi elliptic coordinates can be effectively applied.

In this paper, we will deal with data of full rank. We say that a given system of points in $\mathbb R^k$ is of full rank if these points are not contained in a hyperplane. Let us immediately formulate some of the main outcomes of this paper. For a given system of points in $\mathbb R^k$ of full rank, we construct an explicit pencil of confocal quadrics   with the following properties:

(i) All the hyperplanes for which the hyperplanar moments of inertia (defined in Section \ref{HMI}) for the given system of points are equal, are tangent to  the same quadrics
from the pencil of quadrics. See Section \ref{s4} and Theorem \ref{nd}.  As an application, we develop regularization procedures for the orthogonal least square methods, analogues of lasso and ridge methods from linear regression, see Section \ref{sec:regular}. This is based on a dual version
of Theorem \ref{nd}, which is given as Theorem \ref{dualnd}. Another motivation for this study comes from the gradient descent methods in machine learning. An optimization algorithm may not be guaranteed to arrive at the minimum in a reasonable amount of time. As pointed out in e.g. \cite{ML}, it often reaches some quite low value of the cost function,  equal to some value $s_0$, quickly enough to be useful. Here we deal with the hyperplanar moment as the cost function, in application to  orthogonal least squares. From Theorem \ref{nd}, we know that the hyperplanes which generate the hyperpanar moment equal to $s_0$ are all tangent to the given quadric from
the confocal pencil of quadrics, where the pencil parameter is determined through the value $s_0$.

(ii) For any given point $P$, among all the hyperplanes that contain it, the best fit is the tangent hyperplane to the quadric from the confocal pencil corresponding to the maximal Jacobi coordinate of the point $P$. The worst fit among the hyperplanes containing $P$ is the tangent hyperplane to the ellipsoid from the confocal pencil that contains $P$. See Theorem \ref{th:GenP1} from Section \ref{sec:direct}.

Both results (i) and (ii) can be seen as generalizations of the classical result of Pearson on orthogonal regression \cite{Pir}, or in other words, on the orthogonal least square method (see e.g. \cite{CB}). The original result of Pearson is stated below as Theorem \ref{th:P1} in Section \ref{sec:bestfit}. The original Pearson result also initiated Principal Component Analysis (PCA), see e.g. \cite{And}. Some of our results have a natural interpretation in terms of PCA.  The confocal pencil of quadrics provides a universal tool to solve  Restricted Principal Component Analysis, restricted at any given point, which we formulate and solve in Section \ref{sec:direct}, see Theorem \ref{th:principal}, Corollary \ref{cor:d23orthcond}, and Example \ref{exam2dim}.  Our generalizations of the Pearson Theorem have natural and important applications in the statistics of the measurement error models, for which  orthogonal regression is known to provide a natural framework, see \cite{CB}, \cite{Full}, \cite{CR}.

We also study classical linear regression from a geometric standpoint, and we do this in a coordinate-free form in Section \ref{sec:dirregel}.
For linear regressions, we provide a geometric characterization of hyperplanes of least squares in a given direction among all hyperplanes that contain a given point. In the case $k=2$, this is done in Theorem \ref{th:dirregcond}, and for  an arbitrary $k$ in Theorem \ref{th:dirregcondk}. All the hyperplanes with the same directional hyperplanar moment of inertia (equal $\mu$) form an ellipsoid.
When $\mu$ varies, the ellipsoids change within a homothetic family of ellipsoids, see Propositions \ref{prop:directhyper} and \ref{prop:homotetic}.

 The results we obtain have applications in restricted regressions, see e.g. \cite{SL} for ordinary linear regressions and see e.g. \cite{Full} and \cite{CR} for orthogonal regressions and  measurement error models. Restricted regressions arise in situations with prior knowledge, which have numerous applications, for example in economics and econometrics, see e.g. \cite{FHJ}. Restricted regressions also appear in tests of hypotheses,  see e.g. \cite{SL, FHJ}.  For precise assumptions about distributional conditions under which test statistics are valid see \cite{Full}. For example, consider a system of $N$ points in $\mathbb R^k$ with the centroid $C$ given. For  restricted orthogonal regression, we derive a new formula for the test statistic for the hypothesis that the hyperplane of the best fit contains a given point $P$:
 \begin{equation}\label{eq:Flambda0}
\frac{N}{N-k+1}(2\lambda_{k_C}-\lambda_{k_P}),
\end{equation}
 where $\lambda_{k_P}$ and $\lambda_{k_C}$ are the maximal Jacobi coordinates respectively of the given point $P$ and the centroid $C$, see Theorem \ref{th:Flambda}. Here, it is assumed that the Jacobi coordinates are induced by the confocal pencil of quadrics associated with the given system of $N$ points. The elegance of the above formula certifies the appropriateness of the geometric methods and tools developed here for their use in this statistics framework.  The developed methods and results are further illustrated in natural statistics settings, see e.g. Examples \ref{ex:cells} and \ref{ex:pressure}.

We will say that an ellipsoid is \emph{rotational} if at least two of its principal semiaxes coincide.

For the construction of the confocal pencil of quadrics associated to a given system of points,  we develop a  method based on the study of points for which the hyperplanar ellipsoid of inertia is  rotational. As we will see in Section \ref{HMI}, the hyperplanar ellipsoid of inertia at a point $O$ determines the moment of inertia for any hyperplane that contains $O$, and vice versa. Hence, the hyperplanar ellipsoid of inertia depends on a point $O$.
 Our construction of the confocal pencil of quadrics in the case of an arbitrary number of dimensions $k$ is studied in Section \ref{s4}, which eventually leads to our main result in this direction, Theorem \ref{nd}.

Let us explain the content of Theorem \ref{nd} in the simplest case of $k=2$ here, see Example \ref{ex:thk2}, see Fig. \ref{fig:2dthm}, and also \cite{DG2022a}. We start from a data set in the plane with full rank and the centroid $C$. We find two points, $F_1$ and $F_2$, symmetric with respect to $C$, that have a circle as the ellipse of inertia for the given data set. The question is what the points $F_1$ and $F_2$ tell us about the initial data set. To answer that question, we construct the pencil of confocal conics with foci $F_1$ and $F_2$. We discover that the pencil of confocal conics has the following property with respect to the initial data set: \emph{the lines in the plane that have the same moment with respect to the given data set are all tangent to the same conic from the confocal pencil of conics.}

 Our  method of constructing the confocal pencil of quadrics in arbitrary dimension $k$ associated with a given data set generalizes these planar considerations. It searches for the points
where the ellipsoid of inertia is  rotational and also takes into account the subtle challenges in generalizing the notion of focal points in the plane to spaces of higher dimensions.  The obtained confocal pencil of quadrics does not contain either the hyperplanar ellipsoid of inertia or any of gyrational
hyperplanar ellipsoids of inertia. In statistics terminology, the confocal pencil of quadrics does not contain either the data ellipsoid or the ellipsoid of residuals.  However, the confocal pencil does contain a specially normalized axial ellipsoid of gyration. This normalization is not known {\it a priori} and comes only as a consequence of applying our method described above. Let us also note that the idea of considering the points
where the ellipsoid of inertia is  rotational is more typical for mechanics than for statistics. A dominant practice in statistics is to consider ellipsoids of data or residuals exclusively centered at the centroid.
The cases of given data of lower rank can be treated similarly, considering subspaces of smaller dimension.

\section{Lines and hyperplanes of the best fit to the data set in $\mathbb R^k$}\label{sec:bestfit}

Nowadays, when we are witnessing the explosion of data science it is important to recall the heroic days of formation of statistics and its mathematical foundations.
Ellipses, as two-dimensional quadrics, made a great entrance in statistics in 1886 in the work of Francis Galton \cite{Gal}, which marked
   the birth of the law of regression.

Galton studied the hereditary transmissions and in \cite{Gal} he focused on the height. He collected the data of 930 adult children and 205 of their respective parentages. To each pair of parents he assigned a ``mid-parent" height, as a weighted average of the heights of the parents.
He established the average regression from mid-parent to offsprings and from offsprings to mid-parent. He formulated the law of regression toward
mediocrity:

{\it When Mid-Parents are taller than mediocrity, their Children tend to be shorter than they. When Mid-Parents are shorter than mediocrity, their Children tend to be taller than they.}

This is how the term {\it regression}, thanks to Galton, entered into statistics, although the method of least squares which is in the background, existed in mathematics from the beginning of the XIX century and works of Gauss and Legendre.

In the course of his investigation, Galton discovered a remarkable role of ellipses in analysis of such data. Galton explained in a very nice manner his line of thoughts and actions. We present that in his own words:

``...{\it I found it hard at first to catch the full significance of the entries in the table...They came out distinctly when I `smoothed' the entries by writing at each intersection of a horizontal column with a vertical one, the sum of entries of four adjacent squares... I then noticed that lines drawn through entries of the same value formed a series of concentric and similar ellipses. Their common center ... corresponded to $68\frac{1}{4}$ inches. Their axes are similarly inclined. The points where each ellipse in succession was touched by a horizontal tangent, lay in a straight line inclined to the vertical in the ratio of $\frac{2}{3}$; those where they were touched by a vertical tangent
lay in a straight line inclined to the horizontal in the   ratio of $\frac{1}{3}$. These ratios confirm the values of average
regression already obtained by a different method, of $\frac{2}{3}$ from mid-parent to offspring, and of $\frac{1}{3}$
from offspring to mid-parent...These and other relations were evidently a subject for mathematical analysis and verification...I noted these values and phrased the problem in abstract terms such as a competent mathematician could deal with, disentangled from all references to heredity, and in that shape submitted it to Mr. Hamilton Dixson, of St. Peter's College, Cambridge...

I may be permitted to say that I never felt such a glow of loyalty and respect towards the sovereignty
and magnificent sway of mathematical analysis as when his answer reached me, confirming, by purely mathematical reasoning, my various and laborious statistical conclusions with far more minuteness than I had dared to hope... His calculations corrected my observed value of mid-parental regression
from $\frac{1}{3}$ to $\frac{6}{17.6}$, the relation between the major and minor axis of the ellipse was changed 3 per cent. (it should be as $\frac{\sqrt 7}{\sqrt 2}$)...}"

\begin{figure}[h]\centering
\includegraphics[width=4cm,height=4cm]{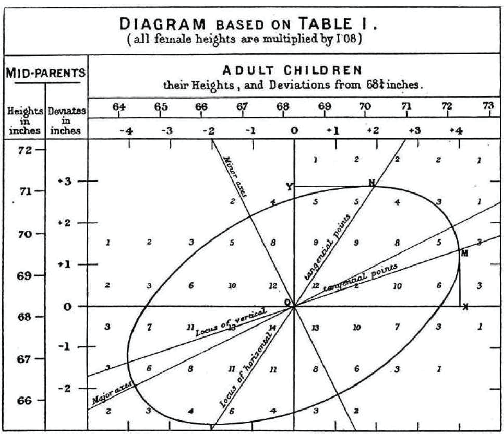}
\caption{From \cite{Gal}.}
\label{fig:Galton}
\end{figure}

In his seminal paper \cite{Pir}, one of the founding fathers of modern statistics, Karl Pearson, investigated the question
of the hyperplane which minimizes the mean square distance from a given set of points in $\mathbb R^k$, for any $k\ge 3$.
In his own words, Pearson formulated the problem: {\it ``In the case we are about to deal with, we suppose
the observed variables--all subject to error--to be plotted
in the plane, three-dimensioned or higher space, and we endeavour
to take a line (or plane) which will be the `best fit' to such
a system of points.
Of course the term `best fit' is really arbitrary; but a
good fit will clearly be obtained if we make the sum of the
squares of the perpendiculars from the system of points upon
the line or plane a minimum."}

\begin{figure}[h]\centering
\begin{minipage}{0.49\textwidth}\centering
\includegraphics[width=4cm,height=3cm]{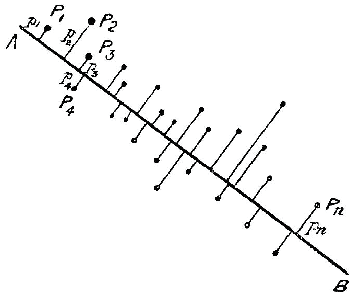}
\caption{From \cite{Pir}.}
\end{minipage}
\begin{minipage}{0.49\textwidth}\centering
\includegraphics[width=3cm,height=3cm]{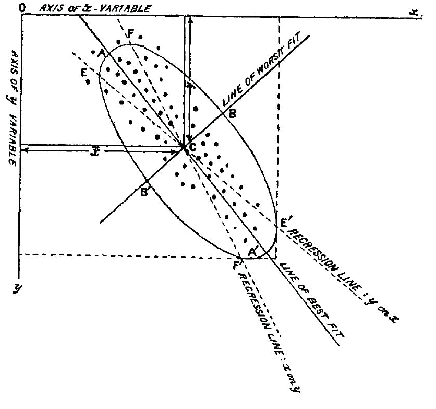}
\caption{From \cite{Pir}.}
\label{fig:Pir1}
\end{minipage}
\end{figure}

Thus, the notion of the best fit is not uniquely determined.  In statistics,  the choice of the squares of the perpendiculars is natural in the measurement error models, or in other words in regression with errors in variables (EIV), \cite{CB}. This corresponds to Pearson's above remark
``we suppose the observed variables--all subject to error".
In such models it is assumed that both predictors and responses are known with some error. In contrast to that, in the classical linear regressions, mentioned above, it is assumed that only responses are known with some error, while for the predictors the exact values are known. Thus,
the squares of distances along one of the axes  is in use in such classical regression models, see also Section \ref{sec:dirk}.
We will call that \emph{directional regression} in a given direction. We will talk more about its geometric aspects in the last Section \ref{sec:dirregel}.

Let us make a more rigorous definitions of the classical simple linear regression model and regression with error in variables (EIV) models, \cite{CB}. For classical simple regression model, it is assumed that the values $(x^{(i)})_{i=1}^N$ are known, fixed values, as for example values set up in advance in the experiment. The values $(y^{(i)})_{i=1}^N$ are observed values of uncorrelated random variables $Y_i$, $i=1, \dots, N$ with the same variance $\sigma^2$.
There is a linear relationship assumed between the predictors $x^{(i)}$ and responses  $(y^{(i)})_{i=1}^N$:
$
EY_i=\alpha +\beta x^{(i)}, \quad i=1, \dots, N.
$
This can be restated as
$
Y_i=\alpha +\beta x^{(i)} + \epsilon_i, \quad i=1, \dots, N,
$
where $\epsilon_i$ are called {\it the random errors} and they are uncorrelated random variables with zero expectation and the same variance $\sigma^2$. In such models the regression is of $Y$ on $x$, i.e. in the vertical direction, see  Example \ref{ex:slr}.
This model was in the background of the
Galton study \cite{Gal}, mentioned above.

There are also important situations where predictors are known only up to some error and they are described by measurement error models. There
the observed pairs $(x^{(i)}, y^{(i)} )_{i=1}^N$ are sampled from random variables $(X_i, Y_i)$ with means satisfying the linear relationship
$
EY_i=\alpha +\beta(EX_i), \quad i=1, \dots, N.
$
Denoting $EX_i=\xi_i$, the errors in variables model can be defined as
$
Y_i=\alpha +\beta\xi_i +\epsilon_i, \quad X_i=\xi_i + \delta_i, \quad i=1, \dots, N,
$
where both $X_i$ and $Y_i$ have error terms which belong to mean zero normal distributions, such that all $\epsilon_i$, $i=1, \dots, N$ have the same variance $\sigma^2_{\epsilon}$ and all $\delta_i$, $i=1, \dots, N$ have the same variance $\sigma^2_{\delta}$. Since in such models
there is a symmetry between $x_i$ and $y_i$ as they are both known with an error, it is more natural to apply to them the orthogonal regression,
or in other words, the orthogonal least square method. This we are going to introduce next under the assumption that $\eta= \sigma^2_{\epsilon}/ \sigma^2_{\delta}=1$   and for an arbitrary numbers of dimensions $k$. One should also mention that applications of the orthogonal least square method in the models with measurement errors have limitations, see e.g. \cite{CR}, \cite{CR0}. These limitations are related to a potentially unknown value of $\eta$. Here we assume that $\eta$ {\it is known}. At first we assume that $\eta=1$. The assumption that $\eta$ is known is essential, while that it is equal to one is not.  If $\eta$ is known, but not equal to 1, it can be made equal to 1 by rescaling either $X$ or $Y$. See also Remark \ref{rem:errorcov}.

From a historic perspective, the case $\eta=1$ originated from \cite{Ad1, Ad2}. Then it was Pearson who established
 \emph{orthogonal regression}  by selecting the squares of the perpendiculars, which corresponds to the case $\eta=1$. Nowadays, it is also called {\it the orthogonal least square method}, see e.g. \cite{CB}, as mentioned above. The geometric aspects of the orthogonal regression are the main subject of this study. We also adopt the Pearson's
generality assumption, that the given system of points does not belong to a hyper-plane i.e. a system of points is of full rank.

To fix the idea, suppose the system of $N$ points $(x_1^{(i)}, x_2^{(i)}, \dots, x_k^{(i)})_{i=1}^N$ is given. Define \emph{the centroid}, or the mean values of the coordinates $\bar x_j$ and \emph{the variances} $
\sigma^2_{x_j}$:
$$
\bar x_j=\frac{1}{N}\sum_{i=1}^Nx_j^{(i)},\ \  \sigma^2_{x_j}=\frac{1}{N}\sum_{i=1}^N(x_j^{(i)}-\bar x_j)^2, \ \  j=1, \dots, k.
$$
Due to the  full rank assumption, all $\sigma^2_{x_j}$, for $j=1, \dots, k$ are non-zero.
Then, \emph{the correlations} $r_{jl}$ and \emph{the covariances} $p_{jl}$ are
$$
r_{jl}=\frac{p_{jl}}{\sigma_{x_j}\sigma_{x_l}},\, p_{jl}=\frac{1}{N}\sum_{i=1}^N(x_j^{(i)}-\bar x_j)(x_l^{(i)}-\bar x_l), j,l=1, \dots, k, l\ne j.
$$
\emph{The covariance matrix} $K$ is a $(k\times k)$ matrix with  diagonal elements
$
K_{jj} = \sigma^2_{x_j}, j=1, \dots, k,
$
and  off-diagonal elements
$
K_{jl}=p_{jl}, j,l=1, \dots, k, l\ne j.
$
The covariance matrix is always symmetric positive semidefinite. However, in this case, we have more: $K$ is a positive-definite matrix due to the full rank assumption. In particular, it has the inverse $K^{-1}$ and all its eigenvalues are positive. Assuming, as it is customary in statistics, that the origin of the Cartesian coordinate system coincides with the centroid,
Pearson defined \emph{the ellipsoid of residuals} by the equation
\begin{equation}\label{eq:residual}
\sum_{j,l=1}^kK_{jl}x_jx_l=const.
\end{equation}
 As we will mention at the end of this Section, the Pearson ellipsoid of residuals is dual to the data ellipsoid.

\begin{remark}\label{pirrem}
Pearson observed that the radius of the ellipsoid of residuals in any direction is equal to the inverse of the standard deviation of the orthogonal projection of the points onto a
 hyperplane that is orthogonal to the line in the direction of the given radius.
\end{remark}

Denote the eigenvalues of $K$ as $\mu_1\ge\dots\ge\mu_k>0$.
\begin{thm}\label{th:P1}[Pearson, \cite{Pir}] The minimal mean square distance from a hyperplane to the given set of $N$ points
is equal to the minimal eigenvalue of the covariance matrix $K$. A best-fitting hyperplane contains the centroid and it is
orthogonal to the corresponding eigenvector of $K$. Thus, it is the principal coordinate hyperplane of the ellipsoid of residuals which is normal to
the major axis.
\end{thm}

Then Pearson studied the lines which best fit to the given set of points and proved
\begin{thm}\label{th:P2}[Pearson, \cite{Pir}]
The line which fits best the given system of $N$ points contains the centroid and coincides with the minor axis of the ellipsoid of residuals.
\end{thm}

Pearson integrated the visualization of  linear regression with  orthogonal regression in the planar case in \cite{Pir}, as shown in
 Fig. \ref{fig:Pir1}. The ellipse in Fig. \ref{fig:Pir1} is dual to the ellipse of residuals, coinciding with the object
studied by Galton. The main goals of this paper are to generalize the aforementioned  classical results of Pearson in the following directions.

{\bf The first goal:} {\it For a given system of $N$ points in $\mathbb R^k$, where $k\ge 2$, and under the  full rank assumption, we consider
all hyperplanes which equally fit  the given system of points. In other words, for any fixed value $s$ that is not less than the smallest eigenvalue $\mu_k$
 of the covariance matrix $K$, we consider all hyperplanes for which the mean sum of square distances to the given set of points is equal to $s$.
Starting from the ellipsoid of residuals,  we  effectively construct a pencil of confocal quadrics with the following property:
For each $s\ge \mu_k$ there exists a quadric from the confocal pencil that serves as the envelope of all the hyperplanes which $s$-fit the given system of points.}

 We stress that  neither the ellipsoid of residuals nor its dual, the data ellipsoid,  belongs to the confocal family of quadrics. The construction of this confocal pencil of quadrics is fully effective, though quite involved.  The obtained pencil
of confocal quadrics is going to have the same center as the ellipsoid of residuals, and moreover, the same principal axes.

\begin{exm}\label{exm:Pearsing} Let us recall that $\mu_k$ denotes the smallest eigenvalue of the covariance matrix $K$. In the case $s=\mu_k$, there is only one hyperplane that fits $s$ to the given set of $N$ points. This is the best-fitting hyperplane described in Theorem \ref{th:P1}. The envelope of this single hyperplane is the hyperplane itself. This hyperplane is going to be a degenerate
quadric from our confocal pencil of quadrics.
\end{exm}

{\bf The second goal:} {\it For a given system of $N$ points in $\mathbb R^k$, where $k\ge 2$, and under the  full rank assumption, find the best fitting hyperplane  under the condition that it contains a selected point in $\mathbb R^k$. We also provide an answer to the questions of the best fitting line
under the condition that it contains a given point.}

A careful look at  Galton's figure (see Fig. \ref{fig:Galton}) reveals an intriguing geometric fact that the line of linear regression
of $y$ on $x$ intersects the ellipse at the points of vertical tangency, while the line of linear regression of $x$ on $y$ intersects the ellipse
at the points of horizontal tangency. Further analysis of this phenomenon leads us to our third goal.

 {\bf The third goal:} to study linear regression in $\mathbb R^k$ in a coordinate free, invariant, form, see Theorem \ref{th:dirregcondk} and Corollary \ref{cor:k2dirreg}. We
address the following question: {\it for a given direction and a given system of $N$ points under the generality assumption, what is the best fitting hyperplane in  the given direction among those that contain a selected point in $\mathbb R^k$?
}

 Apparently, the second and the third goal are addressed using the same confocal pencil of quadrics, which is constructed in relation to the first goal and mentioned above.

 Let us conclude the introduction by observing that the ellipsoids reciprocal to the Pearson ellipsoid of residuals, known as \emph{the data ellipsoids},  have been studied in statistics, as seen in \cite{Cr}. Detailed explanation of the data ellipsoid and the concentration ellipsoid will be given in Section \ref{sec:dict}.

The notions of moments in statistics came from mechanics, where they were originally introduced in  three-dimensional space.
We review these basic notions from a mechanics perspective in the Supplementary material.

\section{The envelopes of hyperplanes which equally fit the data in $k$-dimensional case}\label{s4}

In this Section, we construct a pencil of confocal quadrics associated to a system of points in $\mathbb R^k$.
This pencil of confocal quadrics is going to be instrumental in achieving of all three of our main goals as listed in Section \ref{sec:bestfit}.
In Section \ref{sec:kconstruction}, we present two main steps in our  method of constructing the confocal pencil of quadrics: starting from a system of points, we consider the central hyperplanar ellipsoid of inertia to which
we {\it attach} the points for which the ellipsoid of inertia is  rotational, as per Definition \ref{dfn:attach}. Then, using the obtained list of the attached points, we {\it assign to} them a confocal pencil of quadrics, in accordance with Definition \ref{def:assigned}.
This eventually leads to the main result in this direction, Theorem \ref{nd}.

In Section \ref{sec:regular}, we deal with the first applications of Theorem \ref{nd} in the regularization for orthogonal regression and in the gradient descent method in machine learning. The initial step is a dual formulation of Theorem \ref{nd}, which is given as Theorem \ref{dualnd}.
Following that, methods of regularization for orthogonal regression,  analogues to the lasso and  ridge for linear regression, are presented.

 We initiate this Section by recalling the basic notions related to hyper-planar moments of inertia and their corresponding operators.

\subsection{Hyperplanar moments of inertia}\label{HMI}

In rigid body dynamics,  axial moments of inertia play an important role (see e.g. \cite{ArnoldMMM}).
Let us recall that for a given line $u\subset \mathbb R^3\{x,y,z\}$, \emph{the axial moment of inertia} $I_u$ is defined as $I_u=\sum\limits_{j=1}^{N}m_jd_j^2$,
where $d_j$ is the distance from the point $M_j$ to the line $u$, for $j=1, \dots, N$.

However, here we are studying the hyperplanar moments of inertia. The details about the axial moments of inertia, their definition in higher dimensions, as well as their connection with the hyperplanar moments of inertia, are provided in the
Supplement material, see also \cite{DG2022a}.

Now we pass to hyperplanar moments of inertia.

Let the system of points $M_1,...,M_N$ with masses $m_1,...,m_N$ be given in $\mathbb{R}^k$. Consider a hyperplane $\pi$ in the same space $\mathbb{R}^k$. The hyperplanar moment of inertia for the system of points for the hyperplane $\pi$ is, by definition:
$
J_{\pi}=\sum\limits_{i=1}^Nm_id_i^2,
$
 where $d_i$ is the  perpendicular distance form the $i$-th point to the hyperplane.
  The Huygens-Steiner Theorem (see e.g. \cite{ArnoldMMM}) gives a relationship between the axial moments of inertia for two parallel axes, where one of them contains the center of masses.
In the case of hyperplanar  moments of inertia, a generalization of the Huygens-Steiner
theorem can be formulated as follows:
$J_{\pi}=J_{\pi_1}+md^2,$
where it is assumed that the hyperplane $\pi_1$ contains the center of masses, while $\pi$ is a hyperplane parallel to $\pi_1$ at the distance $d$.
Here $m$ is the total mass of the system of points. The hyperplanar operator of inertia at the point $O$ is defined here as a $k$-dimensional symmetric operator as follows:
$
\langle J_O\mathbf{n_1},\mathbf{n_2}\rangle=\sum\limits_{j=1}^{N}m_j\langle\mathbf{r_j},\mathbf{n_1}\rangle\langle\mathbf{r_j},\mathbf{n_2}\rangle,
$
where $\mathbf{r_j}$ is the radius vector of the point $M_j$.
We see that
$
J_\pi=\langle J_O\mathbf{n},\mathbf{n}\rangle,
$
where $\mathbf{n}$ is the unit vector orthogonal to the hyperplane $\pi$ which contains $O$.  The diagonal elements of the hyperplanar inertia operator are the moments of inertia for the coordinate planes: $J_{ii}=\sum_{j=1}^Nm_j(x_{i}^{(j)})^2$ is the hyperplanar moment of inertia for the plane orthogonal to coordinate axis $Ox_i$, where $x_{i}^{(j)}$ is the $i$-th coordinate of point $M_j$.  The nondiagonal elements of the hyperplanar inertia operator are also called \emph{the centrifugal hyperplanar moments of inertia}: $J_{il}=\sum_{j=1}^Nm_jx_{i}^{(j)}x_{l}^{(j)}$.

In a principal basis, i.e. one  where the hyperplanar operator of inertia has a diagonal form, the diagonal elements are called \emph{the principal hyperplanar moments of inertia}, and they are the hyperplanar moments of inertia for  the coordinate
hyperplanes:
$
J_1=\sum\limits_{j=1}^Nm_j(x_{1}^{(j)})^2,\, J_2=\sum\limits_{j=1}^Nm_j(x_{2}^{(j)})^2,\,\dots,\, J_k=\sum\limits_{j=1}^Nm_j(x_{k}^{(j)})^2
$. (In principal coordinates the centrifugal moments are zero).
The hyperplanar ellipsoid of inertia at the point $O$ is the ellipsoid
$$
\langle J_Ou,u\rangle=1,\quad u\in \mathbb{R}^k.
$$
 Let $X$ be a point on the hyperplanar ellipsoid of inertia centered at $O$. Let $\pi$ be a hyperplane that contains $O$ and is orthogonal to $OX$. From the definition of the hyperplanar moment of inertia, it follows that:
\begin{equation}\label{eq:jpi}
J_{\pi}=\langle J_O\frac {\mathbf{OX}}{||\mathbf{OX}||}, \frac{\mathbf{OX}}{||\mathbf{OX}||}\rangle=\frac{1}{||\mathbf{OX}||^2}.
\end{equation}
Thus, the reciprocal value of the square of the distance $OX$  is equal to the hyperplanar moment of inertia for the hyperplane that contains $O$ and that is orthogonal to $OX$. This implies that the ellipsoid of inertia at a point $O$ completely determines the moments of inertia at  point $O$. This also means that the hyperplanar ellipsoid of inertia depends on $O$.

In principal coordinates, the hyperplanar ellipsoid of inertia at  point $O$ has the form  $J_1x_1^2+...+J_kx_k^2=1$. The ellipsoid given by
$$
\langle J^{-1}_Ou,u\rangle=1, u\in \mathbb{R}^k.
$$
is referred to us as \emph{the  hyperplanar ellipsoid of gyration  at  point $O$}.

We also introduce the following definition: the ellipsoid given by
$$\langle J^{-1}_Ou,u\rangle=m^{-1}, u\in \mathbb{R}^k$$
is called \emph{the mass normalized  hyperplanar ellipsoid of gyration  at point $O$}, where $m$ is the sum of masses of all points from the given dataset.

 An important special case is when point $O$ coincides with the center of masses i.e. when $O$ is the centroid $C$. Then, the hyperplanar operator of inertia at $C$ is called \emph{the central hyperplanar operator of inertia}, the principal hyperplanar moments of inertia at $C$ are \emph{the central principal hyperplanar moments of inertia}, and the hyperplanar ellipsoid of inertia at point $C$ is \emph{the central hyperplanar ellipsoid of inertia}. Similarly, the hyperplanar ellipsoid of gyration at  point $C$ and the mass normalized hyperplanar ellipsoid of gyration at  point $C$ is called \emph{the central hyperplanar ellipsoid of gyration} and \emph{the central mass normalized hyperplanar ellipsoid of gyration}, respectively.

\
\subsection{Planar case and axial moments}\label{sec:planaraxial}

The planar case ($k=2$) is specific since the axial and planar moments of inertia are both defined with respect to lines. The axial moment of inertia and the planar moment of inertia for a line $\ell$ in the plane coincide.  By definitions, they are the sum of the products of the masses of points $M_i$ and the square of the distances from the points $M_i$ to $\ell$. However, the associated inertia operators are not the same: one can check that in the standard basis of  Cartesian coordinates, the matrix of the planar inertia operator is the adjugate matrix of the matrix of the axial inertia operator. Hence, for example, one has that $I_{11}=J_{22}$. Thus, the axial moment of inertia for $Ox$-axis is $I_{11}=\sum_{i=1}^N m_i y_i^2$, and the planar moment of inertia for $O x$-axis is $J_{22}=\sum_{i=1}^N m_i y_i^2$.

In the $k$-dimensional case, one can define the intermediate  moments of inertia of $s$-dimensional planes for $1\le s <k$ as well. They are studied in \cite{DG2022a}. Specifically, for $s=1$, one gets an axial moment of inertia for the lines in $k$-dimensional space see Supplement material, Section 1. The hyperplanar moments of inertia correspond to the case $s=k-1$.

\subsection{A dictionary between mechanical and statistical terminologies}\label{sec:dict}

In order to make the exposition of the paper more accessible to both statistically and mechanically oriented communities, we will provide a brief dictionary in the following, interrelating the notions from Sections \ref{sec:bestfit} and \ref{HMI}.  To that end, we will suppose here that all $m_i$ for $1\leq i \leq N$ are equal (by setting $m_i=1$ for $1\leq i \leq N$ or $m_i=1/N$ for $1\leq i \leq N$). Then the mechanical terms discussed above have their statistical counterparts. We should also note though that having $m_i$ not all being equal also carries significance not only in mechanics but in statistics as well, where also the concept of weights is sometimes applied.

 The matrix of inertia of the operator $J_O$   coincides with the well-known notion in statistics of SSCP (sum of squares and cross-products) of the radius vectors. The  principal moments of inertia are its eigenvalues.
 If $O$ coincides with the centroid $C$, then the matrix of the central hyperplanar inertia operator
is the covariance matrix $K$, and the central hyperplanar ellipsoid of inertia is the ellipsoid of residuals.  The data ellipsoid  becomes the hyperplanar ellipsoid of gyration. The hyperplanar moment of inertia $J_{jj}$ for the coordinate hyperplane orthogonal to the axis $Ox_j$ is the variance $\sigma_{x_j}^2$, while the centrifugal hyperplanar moments of inertia $J_{ij}$ are covariances $p_{ij}$. In a central principal coordinate system, one gets that the central principal moments of inertia $J_{1},...,J_{k}$ coincide
with the eigenvalues of $K$, $\mu_k,...,\mu_1$.
If a hyperplane $\pi$ contains $C$, then the hyperplanar moment of inertia $J_{\pi}$ is the sum of the squared distances from the points of the given system to the hyperplane $\pi$. In the statistical terminology, it represents the variance of the orthogonal projection onto $\pi$. Thus, one can view Pearson's note given in Remark \ref{pirrem} as a statistical
interpretation of the relation \eqref{eq:jpi}.

 In \cite{Cr} the concentration ellipsoid for a given system of points in $\mathbb{R}^k$ is defined as the ellipsoid with the following property: the uniform distribution inside the concentration ellipsoid has the same first and second order moments as the given system of points. The concentration ellipsoid has the equation \cite{Cr}:
 $$
\sum_{j,l=1}^kK^{-1}_{jl}x_jx_l=k+2.
$$
When the constant $k+2$ is replaced with an arbitrary positive constant $c$, a family of homothetic ellipsoids is obtained. This family coincides with the ellipsoids homothetic to the hyperplanar ellipsoid of gyration defined below (see\cite{Cr} for the planar case). In the planar case when $k=2$ the ellipsoid of concentration becomes the ellipse of concentration  $\sum_{j,l=1}^2K^{-1}_{jl}x_jx_l=4$. As explained in Section \ref{sec:planaraxial}, the planar ellipse of inertia and axial ellipse of inertia are dual to each other: the matrix of the planar inertia operator is the adjugate matrix of the matrix of the axial inertia operator. In the planar case, the family of ellipses homothetic to the concentration ellipse coincides with the family of ellipses homothetic to the axial ellipses of inertia, as it is pointed out in \cite{Cr}.
In our paper, the ellipsoid from the family of homothetic ellipsoids, given by
 $$
\sum_{j,l=1}^kK^{-1}_{jl}x_jx_l=1.
$$
will be also referred as the data ellipsoid.

 The above Pearson results can be naturally restated in terms of the data ellipsoid as well. In particular, in the two-dimensional case, the Galton ellipses
 are examples of data ellipsoids,  and they are, as we mentioned above, dual to ellipses of residuals of Pearson. The data ellipsoids in the two-dimensional case are naturally called the data ellipses; they are presented in Fig. \ref{fig:Galton} and \ref{fig:Pir1}.

\

\subsection{Construction of the confocal pencil of quadrics associated to the data in $\mathbb R^k$}\label{sec:kconstruction}

A confocal pencil of conics in plane is fully determined by two points: the common focal points of all the conics from the pencil.
It is quite challenging to present any  analogous  ``focal set" in arbitrary dimension that would geometrically define a confocal pencil of quadrics through some sort of a rope construction (see   \cite{DG2022a} and references therein). For our purposes, we consider the following set (\cite{DG2022a}).

\begin{dfn}\label{def:assigned}
Let a $(k-1)$-dimensional ellipsoid be given by
\begin{equation}\label{kdimelipsoid}
\frac{x_1^2}{\alpha_1}+...+\frac{x_k^2}{\alpha_k}=1,
\end{equation}
where $\alpha_1>...>\alpha_k>0$. We say that the set of
 $2k-2$ collinear \emph{generalized focal points}
 $
 F_{1i}(\sqrt{\alpha_1-\alpha_i},\allowbreak 0,...,0)$, $F_{2i}(-\sqrt{\alpha_1-\alpha_i},0,...,0)$, $i=2,...,k$,
 organized in $(k-1)$ pairs of points symmetric with respect to one given point is \emph{assigned to} the ellipsoid \eqref{kdimelipsoid}.
\end{dfn}
All these assigned points belong to the coordinate axis $Ox_1$, which contains the major semi-axis of the ellipsoid and remains the same for all quadrics from the confocal pencil
\begin{equation}\label{kdimelipsoidpencil}
\frac{x_1^2}{\alpha_1-\lambda}+...+\frac{x_k^2}{\alpha_k-\lambda}=1.
\end{equation}

Let a dataset of full rank in $\mathbb R^k$ be given.
Suppose that
the central principal moments of inertia satisfy $0<J_1<J_2<...<J_k$. We define $a^2_2,a^2_3,...,a^2_k$ by
$
J_1+ma_2^2=J_2,\quad J_1+ma_3^2=J_3,...,J_1+ma_k^2=J_k.
$
From the last formulas, we obtain
\begin{equation}\label{Jndim}\begin{aligned}
J_1&=J_k-ma_k^2,\, J_2=J_k+ma_2^2-ma_k^2, \dots,\\
J_{k-1}&=J_k+ma_{k-1}^2-ma_k^2.
\end{aligned}
\end{equation}

Important properties of the points with rotational ellipsoid of inertia were recently studied in \cite{DG2022a}.
\begin{dfn}\label{dfn:attach}
For the central hyperplanar ellipsoid of inertia
\begin{equation}\label{eq:hyperplanarellipsoid}
J_1x_1^2+...+J_kx_k^2=1
\end{equation}
let us \emph{attach}  the points $F_{1i}(a_i,0,...,0)$ and $F_{2i}(-a_i,0,...,0)$, for $i=2,...,k$, where $\pm a_i$ are determined from \eqref{Jndim}.
\end{dfn}

As a consequence of the generalization of the Huygens-Steiner theorem, the attached points from Definition \ref{dfn:attach} are characterized by the following property: they lie on the line of the major axis of the central ellipsoid of inertia, where two principal moments of inertia coincide. In other words, these are the points on the major axis of the central ellipsoid of inertia at which the hyperplanar ellipsoid of inertia becomes  rotational.

 Now, let us construct the confocal pencil of quadrics for which these points {\it attached to} the  central hyperplanar ellipsoid of inertia   are  {\it assigned points} in the sense of Definition \ref{def:assigned}:
\begin{equation}\label{kdimkonfokal}
\frac{x_1^2}{\frac{J_1}{m}-\lambda}+\frac{x_2^2}{\frac{J_1}{m}-a_2^2-\lambda}+...+\frac{x_k^2}{\frac{J_1}{m}-a_k^2-\lambda}=1.
\end{equation}

Thus, we construct the pencil of confocal quadrics \eqref{kdimkonfokal} beginning with the  central  hyperplanar ellipsoid of inertia \eqref{eq:hyperplanarellipsoid} in a highly nontrivial manner: first, we {\it attach} points with a  rotational ellipsoid of inertia, and then we {\it assign} to them the confocal pencil of quadrics. The resulting confocal pencil of quadrics possesses a remarkable property concerning the initial system of points that defined
the hyperplanar ellipsoid of inertia.

\begin{thm}\label{nd} Given a system of points in $\mathbb R^k$ with a total mass $m$ and  central principal hyperplanar moments of inertia  $0<J_1<J_2<\dots<J_k$, along with the center of masses $C$, consider the family of hyperplanes for which the system of points exhibits the same hyperplanar moment of inertia.  The envelope of all hyperplanes within this family, that do not contain the origin, is the quadric from the pencil of confocal quadrics:
\begin{equation}\label{kdconfocal}
\frac{x_1^2}{A^2}+\frac{x_2^2}{A^2-a_2^2}+...+\frac{x_k^2}{A^2-a_k^2}=1,
\end{equation}
where
$$
a_s^2=\frac{J_s-J_1}{m}, \quad s=2, \dots, k.
$$

All hyperplanes within this family,  that contain the origin, are asymptotically tangent at infinity to the same quadric \eqref{kdconfocal} from the pencil of confocal quadrics.
\end{thm}
\begin{proof}
Formula \eqref{kdconfocal} is obtained from \eqref{kdimkonfokal} by substitution
\begin{equation}\label{AprekoJ}
A^2=\frac{J_1}{m}-\lambda.
\end{equation}

Given a hyperplane  $\pi_1$  at the distance $d$ from the center of masses $C$, let  $\mathbf{n(\pi_1)}=(n_1,n_2,...n_k)$ be the unit vector orthogonal to the hyperplane $\pi_1$. The hyperplanar moment of inertia is given by
$$
\begin{aligned}
J_{\pi_1}&=J_1n_1^2+....+J_kn_k^2+md^2\\
&=J_k\sum_{j=1}^kn_j^2
+m(d^2+a_2^2n_2^2+a_3^2n_3^2+...+a_{k-1}^2n_{k-1}^2-a_k^2\sum_{j=1}^{k-1}n_j^2)\\
&=J_k+m(d^2+a_2^2n_2^2+a_3^2n_3^2+...+a_{k-1}^2n_{k-1}^2+a_k^2n_k^2-a_k^2).
\end{aligned}
$$
For all the hyperplanes that share the same hyperplanar moment of inertia, the value
\begin{equation}\label{sk}
S_k(\pi_1)=d^2+a_2^2n_2^2+a_3^2n_3^2+...+a_{k-1}^2n_{k-1}^2+a_k^2n_k^2-a_k^2
\end{equation}
has to remain constant, regardless of the choice of $\pi_1$.

Given an arbitrary tangent hyperplane $\pi_2$ at an arbitrary point $(x_{10},x_{20},...,x_{k0})$ of the quadric \eqref{kdconfocal}, its equation is as follows:
$$
\frac{x_1x_{10}}{A^2}+\frac{x_2x_{20}}{A^2-a_2^2}+...+\frac{x_kx_{k0}}{A^2-a_k^2}-1=0,
$$
where the unit normal vector is
$$
\mathbf{n(\pi_2)}=\frac{1}{\Delta}\Big(\frac{x_{10}}{A^2},\frac{x_{20}}{A^2-a_2^2},...,\frac{x_{k0}}{A^2-a_k^2}\Big),
$$
and
$$
\Delta^2=\frac{x_{10}^2}{A^4}+\frac{x_{20}^2}{(A^2-a_2^2)^2}+...+\frac{x_{k0}^2}{(A^2-a_k^2)^2}.
$$
 Here, ${\Delta}^{-1}$ also represents  the distance from the center of masses $C$ to the  hyperplane $\pi_2$. For $S_k(\pi_2)$ we have:
$$
\begin{aligned}
S_k(\pi_2)&
=\frac{1}{\Delta^2}\Big(1+\frac{a_2^2x_{20}^2}{(A^2-a_2^2)^2}+...+\frac{a_k^2x_{k0}^2}{(A^2-a_k^2)^2}\Big)-a_k^2\\
&=\frac{1}{\Delta^2}\Big(\frac{x_{10}^2}{A^2}+\frac{x_{20}^2}{A^2-a_2^2}+...+\frac{x_{k0}^2}{A^2-a_k^2}+\frac{a_2^2x_{20}^2}{(A^2-a_2^2)^2}+
...+\frac{a_k^2x_{k0}^2}{(A^2-a_k^2)^2}\Big)-a_k^2\\
&=\frac{1}{\Delta^2}A^2\Big(\frac{x_{10}^2}{A^4}+\frac{x_{20}^2}{(A^2-a_2^2)^2}+...+\frac{x_{k0}^2}{(A^2-a_k^2)^2}\Big)-a_k^2=A^2-a_k^2.
\end{aligned}
$$
 If $S_k(\pi_2)>0$, then $S_k(\pi_2)$ is equal to the square of the smallest semi-axis of the quadric \eqref{kdconfocal}.   Similarly, if $S_k(\pi_2)<0$, then $S_k(\pi_2)$ is equal to the negative square of the smallest semi-axis of the quadric \eqref{kdconfocal}.  This implies that $S_k=S_k(\pi)$ remains the same  for those hyperplanes $\pi$ which are tangent to the quadric \eqref{kdconfocal}, where the pencil parameter $A$ is given by:
$$
A^2=S_k+a_k^2= S_k+\frac{J_k-J_1}{m}.
$$
This proves the theorem.
\end{proof}

The type of the enveloping quadric depends on the value of the parameter $A$.
Here, $J_{\pi}=J_k+mS_k=J_k+m(A^2-a_k^2)$, or
$
A^2=\frac{J_{\pi}-J_1}{m}.
$
From the last formula, it follows that hyperplanes with a hyperplanar moment of inertia equal to $J_\pi$ are tangent to the following quadric from the confocal pencil \eqref{kdimkonfokal}:
\begin{equation}\label{eq:Jpim}
\frac{x_1^2}{J_\pi-J_1}+\frac{x_2^2}{J_\pi-J_2}+...+\frac{x_k^2}{J_\pi-J_k}=\frac{1}{m}.
\end{equation}

Thus, in terms of the value of the hyperplanar moment of inertia $J_{\pi}$, one gets the following information about the type of the
enveloping quadric:
\begin{itemize}
\item
If $J_{\pi}>J_k$, then the enveloping quadric is a $k$-dimensional
ellipsoid.
\item
If $J_i<J_{\pi}<J_{i+1}$, then the enveloping quadric belongs to the  $i$-th type quadrics.

\item
In the case $J_\pi<J_1$, there is no hyperplane with a hyperplanar moment of inertia equal to $J_\pi$.
\item
If $J_\pi=J_i$, for any $i=1, \dots, k$, then the enveloping quadric is degenerate, and the hyperplane $\pi$ coincides with the enveloping quadric as a coordinate hyperplane.

\end{itemize}

\begin{exm}\label{ex:thk2} We can consider the case $k=2$ of Theorem \ref{nd}. For a given system of points in $\mathbb R^2$, let $C$ denote the center of masses, and $F_1$ and $F_2$ be the two points with a circle as the ellipse of inertia. Then any conic with the foci $F_1$ and $F_2$ has the property that for any of its tangents, $\pi_1$, $\pi_2$, and $\pi_3$, it holds: $J_{\pi_1}=J_{\pi_2}=J_{\pi_3}$, see Fig. \ref{fig:2dthm}.

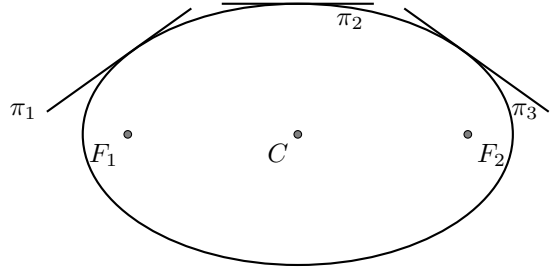
\begin{figure}[h] \centering
\begin{tikzpicture}

\draw[thick](0,0) circle [x radius={sqrt(8)}, y radius={sqrt(3)}];

\draw[black,fill=gray]({sqrt(5)},0) circle [x radius=0.05,y radius=0.05] node[anchor=north west] {$F_2$};
\draw[black,fill=gray]({-sqrt(5)},0) circle [x radius=0.05,y radius=0.05] node[anchor=north east] {$F_1$};
\draw[black,fill=gray](0,0) circle [x radius=0.05,y radius=0.05] node[anchor=north east] {$C$};

\draw[thick] (-1,{sqrt(3)}) -- (1,{sqrt(3)}) node[anchor= north east] {$\pi_2$ };
\draw[thick] (1.4,1.67) -- (3.3,0.3) node[anchor= east] {$\pi_3$};
\draw[thick] (-1.4,1.67) -- (-3.3,0.3) node[anchor= east] {$\pi_1$};

\end{tikzpicture}
\caption{Theorem \ref{nd}, case $k=2$: $J_{\pi_1}=J_{\pi_2}=J_{\pi_3}$}
\label{fig:2dthm}
\end{figure}

\end{exm}

\begin{exm}[\cite{DG2022a}]
Given a system of points of full rank in three-dimensional space with the  central principal planar moments of inertia  $J_1<J_2<J_3$, consider the family of planes for which the system of points has the same planar moment of inertia, equal $J_{\pi}$. All planes from the family are tangent to the same quadric from the pencil of confocal quadrics, which is:
\begin{itemize}
\item[(i)] if $J_{\pi}>J_3$, then the envelope is an ellipsoid;
\item[(ii)] if $J_2<J_{\pi}<J_3$, then  the envelope is a one-sheeted hyperboloid;
\item[(iii)] if  $J_1<J_{\pi}<J_2$, then the envelope is a  two-sheeted hyperboloid;
\item[(iv)] if $J_{\pi}<J_1$, then there is no plane with the planar moment of inertia equal to $J_{\pi}$;
\item[(v)] if $J_{\pi}=J_i$, then for $i=1, 2,3$, the envelope is the corresponding principal coordinate plane.
\end{itemize}
\end{exm}

A trace of the statement of Theorem \ref{nd} in the three-dimensional case can be retrieved from \cite{Ga}.  However, the presentation in  \cite{Ga} was very vague and unusual even by the standards of that historic period. It was written more like a report or an essay rather than as a scientific paper, as it did not provide clear statements, formulas, or  proofs of the main assertions.

\begin{prop}\label{prop:origin} Given a system of points in $\mathbb R^k$ with a total mass $m$ and  central principal hyperplanar moments of inertia  $0<J_1<J_2<\dots<J_k$, along with the center of masses $C$, consider
all hyperplanes $\pi$ that contain the origin
$$
\pi: \beta_1x_1+\dots +\beta_kx_k=0.
$$
All such hyperplanes with the same hyperplanar moment $J_{\pi}=\mu$ for a given value $\mu$ are parameterized as follows:
\begin{equation}\label{eq:orign}
(J_1-\mu)\beta_1^2+\dots(J_k-\mu)\beta_k^2=0.
\end{equation}
\end{prop}

\subsection{Applications in the regularization of the orthogonal least square method}\label{sec:regular}

Using duality between quadrics of points and quadrics of tangential hyperplanes (see, for example, \cite{ArnoldMMM, DR}) , we can reformulate the above Theorem \ref{nd} as follows:

\begin{thm}\label{dualnd} Given a system of points in $\mathbb R^k$ with a total mass of $m$ and  central principal hyperplanar moments of inertia  $0<J_1<J_2<\dots<J_k$, consider the family of hyperplanes for which the system of points has the same hyperplanar moment of inertia. All the hyperplanes from this family  which do not contain the origin form a quadric. By varying the hyperplanar moment of inertia, the resulting quadrics form a linear pencil that is dual to the confocal pencil  \eqref{kdimkonfokal}.
\end{thm}

 This dual formulation,  along with Proposition \ref{prop:origin}, can be used to study nonlinear constraints on the hyperplanes of regression. Such situations may arise in regularization
 problems for the orthogonal least square method. In the case of the regularization of the standard linear regression, there are various shrinkage methods developed (see e.g. \cite{HTF}). For example
 the ridge or the lasso methods depend on the type of the norm used to bound the coefficients of the hyperplanes. Here we develop similar methods for the orthogonal least square. A ridge-type method imposes an $L_2$ bound on the coefficients $\beta_1, \dots, \beta_k$ of the hyperplanes:
 $
 ||\beta||_2\le s.
 $
 A lasso-type method assumes the use of the $L_1$ norm and the condition on the coefficients $\beta_1, \dots, \beta_k$ of the hyperplanes can be written as:
 $
 ||\beta||_1\le t.
 $
 The best-fit hyperplane under each of these conditions is determined as the point of tangency of a quadric from the linear pencil from Theorem
 \ref{dualnd} and Proposition \ref{prop:origin} and the $L_2$ circle of radius $s$ in the first case and the $L_1$ circle (aka the diamond) of radius $t$ in the second case, see Fig. \ref{fig:L1L2}.

\begin{figure}[h]\centering
\includegraphics[width=7.2cm,height=4cm]{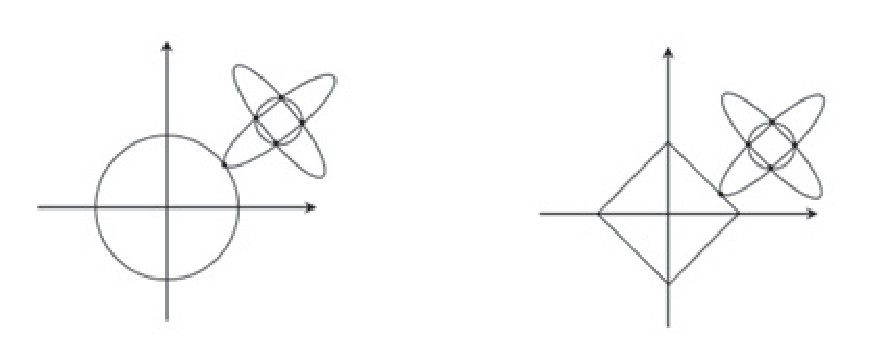}
\caption{Regularization for the orthogonal least square method: on the left, $L_2$ or ridge-type case; on the right, $L_1$ or lasso-type case. See Remark \ref{rem:pencil}.}
\label{fig:L1L2}
\end{figure}

The situation of the lasso and ridge in  classical linear regression was presented in Figure 3.11 of \cite{HTF}. In linear regression, the family of quadrics  presents
a set of homothetic quadrics defined by equation (3.51) from \cite{HTF}, see also Propositions \ref{prop:directhyper} and \ref{prop:homotetic}. On the other hand, in the orthogonal least square situation developed here,  a linear pencil of quadrics is used from Theorem  \ref{dualnd} and Proposition \ref{prop:origin}  We devote separate publications \cite{DG2024a, DG2024b} to a detailed study of geometric aspects of the lasso and its analogue in the orthogonal regression.

\begin{remark}\label{rem:pencil} A linear pencil of quadrics is generated by two quadrics.  In the  complex projective plane, a linear pencil of conics is generated by two conics, as seen in Fig. \ref{fig:L1L2}.
A generic linear pencil of conics consists of all conics that share four common points, see e.g. \cite{DR}.
\end{remark}

Another application of the methods developed in this paper is related to  gradient descent in machine learning. The optimization algorithm may not be guaranteed to arrive at the minimum in a reasonable amount of time. However, as noted in \cite{ML}, it often reaches some quite low value of the cost function (here the hyperplanar moment) equal to $s_0$ quickly enough to be useful. In application to the orthogonal least square, from Theorem \ref{nd}, we know that the hyperplanes which generate the hyperplanar moment equal to $s_0$ are all tangent to the given quadric from the pencil \eqref{kdimkonfokal}; the pencil parameter is determined through $s_0$.

\section{Hyperplanes containing a given point which best fit the data in $\mathbb R^k$. Restricted PCA } \label{sec:direct}

As a further employment of the confocal pencil of quadrics which we constructed in Section \ref{sec:kconstruction}, associated to a given system of points in $\mathbb R^k$, see Theorem \ref{nd}, we are going to reach our second goal as listed in Section \ref{sec:bestfit}. We also formulate and solve the Restricted PCA, restricted at a point, using the same confocal pencil of quadrics as a universal tool, see Theorem \ref{th:principal}. In Example \ref{ex:cells}, we apply the developed methods in a natural situation of an errors-in-variables model. We show how it effectively works
in testing a hypothesis that the line of the best fit contains a given point. We derive a new formula for test statistics in restricted orthogonal least square method in Theorem \ref{th:Flambda}.

Let the system of points $M_1,...,M_N$ with masses $m_1,...,m_N$ be given in $\mathbb{R}^k$ and let $C$ be the center of masses. We consider the following problem: {\it For a given point  $P(x_{01},...,x_{0k})$ in $\mathbb{R}^k$ find the principal hyperplanes of the hyperplanar ellipsoid of inertia at the point $P$. We also find the hyperplane that minimizes the hyperplanar moment of inertia among the hyperplanes that contain $P$}.

The family of confocal quadrics \eqref{kdimkonfokal} which we associated with the given system of points, provides the solution to the above problem in the following way. Assume that the principal axes  at the point $C$ are chosen as the coordinate axes. Let the hyperplane $\pi$  containing the point $P$ be given, with $\mathbf{n}=(n_1,n_2,...,n_k)$, as the unit normal vector to the hyperplane. Using the Huygens-Steiner theorem, the moment of inertia for the hyperplane $\pi$ is
$
J_{\pi}=J_{\pi_C}+md^2,
$
where $J_{\pi_C}$ is the moment of inertia for the hyperplane $\pi_C$, where $\pi_C$ contains the center of masses $C$ and is parallel to the hyperplane $\pi$. Let $d$ be the distance from $C$ to $\pi$. The square of the distance is given by formula
$d^2=(x_{01}n_1+...+x_{0k}n_k)^2.$
One has
$
J_{\pi}=J_1n_1^2+...+J_kn_k^2+md^2=J_1n_1^2+...+J_kn_k^2+m(x_{01}n_1+...+x_{0k}n_k)^2.
$
The direction of the principal axes at the point $P$ can be obtained as the stationary points of the function $J_{\pi}(n_1,...,n_k)$.
Since $n_1^2+...+n_k^2=1$, the stationary points can be obtained using the Lagrange multipliers.
As usual, consider the function
\begin{equation}\label{eq:Lagrange}\begin{aligned}
F(n_1,...,n_k,\mu)&=J_1n_1^2+...+J_kn_k^2+m(x_{01}n_1+...\\
&+x_{0k}n_k)^2-\mu(n_1^2+...+n_k^2-1),
\end{aligned}
\end{equation}
where $\mu$ is the Lagrange multiplier. We get the stationary points from the condition that the partial derivatives of $F$ are equal to zero. One gets the system:
\begin{equation}\label{sistemlambda}
(J_P-\mu E)\mathbf{n}=0, \quad n_1^2+...+n_k^2-1=0,
\end{equation}
where $J_P$ is the planar inertia operator at the point $P$, given by
\begin{equation}\label{JuP}
J_P=\left(\begin{matrix}J_1+mx_{01}^2&mx_{01}x_{02}&...&mx_{01}x_{0k}\\
                     mx_{02}x_{01}&J_2+mx_{02}^2&...&mx_{02}x_{0k}\\
                     .& & &\\
                     .& & &\\
                     .& & &\\
                     mx_{0k}x_{01}&mx_{0k}x_{02}&...&J_k+mx_{0k}^2
\end{matrix}\right).
\end{equation}
The first part in \eqref{sistemlambda} is a homogeneous system of $k$ equations in $n_1,...,n_k$ with $\mu$ as a parameter.
The nontrivial solutions exist under the condition that the determinant of the system
vanishes. This determinant  is equal to the characteristic polynomial of the symmetric matrix $J_P$, with roots $\mu_1,...,\mu_k$ being
the eigenvalues of the matrix $J_P$. They are the principal planar moments of inertia at the point $P$. Each eigenvalue $\mu_j$ has a corresponding eigenvector $\mathbf{n}_{{j}}$, which is the unit normal vector of a principal hyperplane. Since $J_P$ is symmetric, all eigenvalues are real, and corresponding
eigenvectors $\mathbf{n}_{(1)},...,\mathbf{n}_{(k)}$ are mutually orthogonal.

Let us look at the system \eqref{sistemlambda} in another way. First $k$ equations can be rewritten as:
$$\begin{aligned}
\frac{mx_{01}}{\mu-J_1}&=\frac{n_1}{n_1x_{01}+...+n_kx_{0k}},\quad
 \dots,\\
\frac{mx_{0k}}{\mu-J_k}&=\frac{n_k}{n_1x_{01}+...+n_kx_{0k}}.
\end{aligned}
$$
By multiplying $j$-th equation by $x_{0j}$ and  adding all of them, we come to the confocal pencil of quadrics:
\begin{equation}\label{konlambde}
\frac{x_{01}^2}{\mu-J_1}+\frac{x_{02}^2}{\mu-J_2}+...+\frac{x_{0k}^2}{\mu-J_k}=\frac{1}{m}.
\end{equation}
Each eigenvalue $\mu_j$ of the matrix $J_P$ corresponds to  one quadric from \eqref{konlambde} which passes through the point $P$.
Denote by $X_P$ the $k\times N$ matrix of the data related to the point $P$
$$
X_P=\left(\begin{matrix}\sqrt{m_1}(x_1^{(1)}-x_{01})&\sqrt{m_2}(x_1^{(2)}-x_{01})  &...&\sqrt{m_N}(x_1^{(N)}-x_{01})\\
                     \sqrt{m_1}(x_2^{(1)}-x_{02})&\sqrt{m_2}(x_2^{(2)}-x_{02})  &...&\sqrt{m_N}(x_2^{(N)}-x_{02})\\
                     .& & &\\
                     .& & &\\
                     .& & &\\
                     \sqrt{m_1}(x_k^{(1)}-x_{0k})&\sqrt{m_2}(x_k^{(2)}-x_{0k})  &...&\sqrt{m_N}(x_k^{(N)}-x_{0k})
\end{matrix}\right).
$$
Here $(x_1^{(j)}, \dots, x_k^{(j)})$ denotes the coordinates of the point $M_j$, for $j=1, \dots, N.$
We have
$
J_P=X_PX_P^T.
$
We can formulate and solve the restricted principal component analysis (RPCA) restricted at the point $P$ as follows: {\it If $n$ is a norm $1$ vector in $\mathbb R^k$, the variance of $n^TX_P$ is $n^TJ_Pn$. We search for the maximal variance among unit vectors, noncorrelated with the previous ones.}

The determination of the maximal variance
of $n^TX_P$ among all unit vectors leads to the conditional extremum problem exactly as stated above, see \eqref{eq:Lagrange} and \eqref{sistemlambda}.
Thus, the above construction and the confocal pencil of quadrics provide a universal tool to solve Restricted PCA restricted at a given point $P$,
for any point $P$. Finally, we have the following
\begin{thm}\label{th:principal} Let the system of points $M_1,...,M_N$ with masses $m_1,...,m_N$ be given in $\mathbb{R}^k$.  For any point $P(x_{01},...,x_{0k})$ there are $k$ mutually orthogonal confocal quadrics from the confocal pencil \eqref{konlambde} that contain the point $P$. The $k$ tangent hyperplanes to these quadrics are
the principal hyperplanes of inertia at the point $P$. The obtained principal coordinate axes are the principal components solving
Restricted PCA, restricted at the point $P$, i.e. providing the maximum variance among the normalized combinations $n^TX_P$, uncorrelated with previous ones.
\end{thm}

In the three-dimensional case, a similar result for the axial moments of inertia can be found in the classical book by Suslov \cite{Sus}.
Let us observe that the pencils of quadrics \eqref{konlambde} and \eqref{kdimkonfokal} coincide. The correspondence of the pencils is realized through the formula:
\begin{equation}\label{eq:mul}
\mu=2J_1-m\lambda.
\end{equation}

We can reformulate Theorem \ref{th:principal} in the following way:

\begin{thm}\label{th:principal1} Let the system of points $M_1,...,M_N$ with masses $m_1,...,m_N$ be given in $\mathbb{R}^k$ with the associated pencil of confocal quadrics \eqref{kdimkonfokal}.  For any point $P(x_{01},...,x_{0k})$ denote its Jacobi coordinates of the pencil of confocal quadrics \eqref{kdimkonfokal} as $(\lambda_1<\dots<\lambda_k)$.   The $k$ tangent hyperplanes to the quadrics with parameters $(\lambda_1<\dots<\lambda_k)$ are
the principal hyperplanes of inertia at the point $P$.
\end{thm}

As a consequence, we get the following generalization of the Pearson Theorem \ref{th:P1}.

\begin{thm}\label{th:GenP1} Let the system of points $M_1,...,M_N$ with masses $m_1,...,m_N$ be given in $\mathbb{R}^k$ with the associated pencil of confocal quadrics \eqref{kdimkonfokal}.  For any point $P(x_{01},...,x_{0k})$ denote its Jacobi coordinates of the pencil of confocal quadrics \eqref{kdimkonfokal} by $(\lambda_1<\dots<\lambda_k)$.  The pencil parameters $\lambda_1,...,\lambda_{k}$ determine the $k$ orthogonal quadrics from the pencil containing $P$. The tangent planes at the point $P$ to these quadrics are the critical points of the hyperplanar moment with respect the hyperplanes which contain $P$.

The hyperplane which is the best fit to the given system of points among the hyperplanes that contain the point $P$ is the tangent hyperplane to the quadric from confocal pencil \eqref{kdimkonfokal} with the pencil parameter $\lambda_k$.
Similarly, the hyperplane which is the worst fit to the given system of points among the hyperplanes that contain the point $P$ is the tangent hyperplane to the quadric from \eqref{kdimkonfokal} with the pencil parameter $\lambda_1$.

\end{thm}

\begin{figure}[h]\centering
\includegraphics[width=7.2cm,height=4cm]{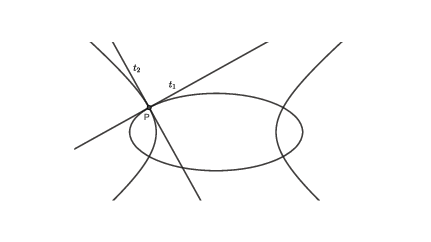}
\caption{With Corollary \ref{cor:d23orthcond} for $k=2$: The ellipse and hyperbola from the confocal pencil passing though $P$. The tangent $t_1$ to the ellipse at $P$ is the worst fit among all the lines containing $P$, while $t_2$, the tangent to the hyperbola at $P$ is the best fit among all such lines.  The tangents $t_1, t_2$ solve RPCA restricted at the point $P$.}
\label{fig:2dorthcond}
\end{figure}

\begin{proof}

The proof follows from Theorem \ref{th:principal1}. The principal hyperplanes of inertia at the point $P$ are tangent hyperplanes to the quadrics from the confocal pencil corresponding to the Jacobi coordinates of the point $P$. The Jacobi coordinates are related to the eigenvalues $\mu$ according to the formula \eqref{eq:mul}. The largest Jacobi coordinate $\lambda_k$ corresponds to the minimal eigenvalue $\mu_1$ and the smallest Jacobi coordinate $\lambda_1$ corresponds  to the largest eigenvalue $\mu_k$.  Comparing  formulas \eqref{konlambde} and \eqref{eq:Jpim}, we see that the smallest eigenvalue $\mu_1$ corresponds to the minimal hyperplanar moment $J_{\pi_1}$. Similarly, the largest eigenvalue $\mu_k$ corresponds to the maximal hyperplanar moment $J_{\pi_k}$. Thus the smallest Jacobi coordinate $\lambda_1$ corresponds to  the maximal hyperplanar moment $J_{\pi_k}$ and the largest Jacobi coordinate $\lambda_k$ corresponds to  the minimal hyperplanar moment $J_{\pi_1}$.
Thus, $\pi_1$ is the tangent hyperplane at $P$ to the quadric with the parameter $\lambda_k$, which is the hyperplane of the best fit among all the hyperplanes containing $P$. Similarly, $\pi_k$ is the tangent hyperplane at $P$ to the quadric with the parameter $\lambda_1$. This quadric is an ellipsoid and its tangent hyperplane at $P$  is the hyperplane of the worst fit among all the hyperplanes containing $P$.
\end{proof}

We now formulate low-dimensional specializations of the last Theorem \ref{th:GenP1}.

\begin{cor}\label{cor:d23orthcond} Let the system of points $M_1,...,M_N$ with masses $m_1,...,m_N$ be given in $\mathbb{R}^k$ with the associated pencil of confocal quadrics \eqref{kdimkonfokal}. Given a point $P$.
\begin{itemize}
\item  For $k=2$ let the point $P$ have the Jacobi coordinates $(\lambda_1<\lambda_2)$. The line which is the best fit for the given system
of points among the lines which contain $P$ is the tangent line to the hyperbola with the parameter $\lambda=\lambda_2$. The line which is the worst fit for the given system
of points among the lines which contain $P$ is the tangent line to the ellipse with the parameter $\lambda=\lambda_1$.   The tangents $t_1, t_2$ solve RPCA restricted at the point $P$. (See Fig. \ref{fig:2dorthcond}.)
\item  For $k=3$ let the point $P$ have the Jacobi coordinates $(\lambda_1<\lambda_2<\lambda_3)$. The plane which is the best fit for the given system
of points among the planes which contain $P$ is the tangent plane  to the two-sheeted hyperboloid with the parameter $\lambda=\lambda_3$. The plane which is the worst fit for the given system
of points among the planes which contain $P$ is the tangent plane to the ellipsoid with the parameter $\lambda=\lambda_1$.  The lines of intersection of each two of the three tangent planes solve RPCA restricted at the point $P$.
\end{itemize}
\end{cor}
\begin{proof}  Let us provide the proof in the case $k=2$.
The matrix of the planar inertia operator at the point $P$ in given coordinates  has the form
$$J_P=\left(
\begin{matrix}
J_x&J_{xy}\\
J_{xy}&J_y
\end{matrix}
\right),
$$
where
$$\begin{aligned}
J_x&=\sum\limits_{i=1}^Nm_i(x_i-x_P)^2,\,
J_y=\sum\limits_{i=1}^Nm_i(y_i-y_P)^2,\\
J_{xy}&=\sum\limits_{i=1}^Nm_i(x_i-x_P)(y_i-y_P).
\end{aligned}
$$
The extremal planar moments of inertia are the eigenvalues of the matrix $J_P$, while the corresponding eigenvectors are orthogonal to the required extremal lines. Thus, $J_{P_1}$ and $J_{P_2}$ are the
solutions of the quadratic equation
$$
\mu^2-(J_x+J_y)\mu+J_xJ_y-J_{xy}^2=0.
$$
One gets that the maximal and the minimal planar moments of inertia are respctively
$$
J_{P_1}=\frac{J_x+J_y+\sqrt{D}}{2},\quad J_{P_2}=\frac{J_x+J_y-\sqrt{D}}{2},
$$
where $D=(J_x-J_y)^2+4J_{xy}$.

After calculating the corresponding eigenvectors, one derives the equations of the extremal lines. The equation of the line of the worst fit among the lines that contain $P$ is
 \begin{equation}\label{eq:worstfitline}
 -2J_{xy}(x-x_P)+(J_x-J_y-\sqrt{D})(y-y_P)=0.
 \end{equation}
 The equation of the line of the that best fit is
 \begin{equation}\label{eq:bestfitline}
(J_y-J_x+\sqrt{D})(x-x_P)-2J_{xy}(y-y_P)=0.
\end{equation}
From
$$
J_{P_1}=2J_{1}-m\lambda_2,\quad J_{P_2}=2J_{1}-m\lambda_1.
$$
we get
$$
\lambda_i=\frac{2J_1-J_{P_j}}{m}, \quad (i,j)\in \{(1,2), (2,1)\}.
$$
The line \eqref{eq:worstfitline} is the tangent at $P$ to the ellipse from the confocal pencil,
which passes through $P$ and has the pencil parameter $\lambda=\lambda_1$.  The line \eqref{eq:bestfitline} is the tangent at $P$ to the hyperbola from the confocal pencil,
which passes through $P$ and has the pencil parameter $\lambda=\lambda_2$.
\end{proof}

\begin{exm}\label{exam2dim}
Let $N$ points with masses $m_1,...,m_N$ be given in the plane. Let us also fix a point $P$ in the plane.
We will give the explicit formulas for the lines that are the best and the worst fit for the given system of $N$ points among the lines that contain the given point $P$, using the confocal pencil of conics \eqref{kdimkonfokal} associated to the given system of $N$ points.
Let $C$ be the centroid of the system of $N$ points. On our way, we provide the solution to the Restricted PCA at the point $P$, see \eqref{eq:n1n2}.

Denote $Cxy$ the principal coordinate system of the confocal pencil and $J_1$ and $J_2$ as before, denote the  principal planar moments of the inertia such that
$J_1+ma_2^2=J_2.$
 The confocal pencil of conics  has the form
\begin{equation}\label{eq:2dimpencil}
\frac{x^2}{\alpha-\lambda}+\frac{y^2}{\beta-\lambda}=1,\quad \alpha=\frac{J_1}{m},\quad \beta=\frac{J_1}{m}-a_2^2.
\end{equation}
For the point $P(x_P,y_P)$, its Jacobi elliptic coordinates $\lambda_1<\lambda_2$ are the solutions of the quadratic equation
\begin{equation}\label{eq:kvadratna}
\lambda^2-(\alpha+\beta-x_P^2-y_P^2)\lambda+\alpha\beta-\beta x_P^2-\alpha y_P^2=0.
\end{equation}

The formula \eqref{eq:mul} gives the connection of the extremal values  $J_{P_1}$ and $J_{P_2}$ of the principal planar inertia operator at the  point $P$ with the Jacobi elliptic coordinates $\lambda_1$ and $\lambda_2$ of the point $P$, associated with the pencil of confocal conics:
$
J_{P_1}=2J_{1}-m\lambda_2,\quad J_{P_2}=2J_{1}-m\lambda_1.
$
We have already mentioned that the line that is  the best fit among the lines that contain the point $P$ is the tangent to the $\lambda_2$-coordinate line through $P$, which is a hyperbola.  The line that is the worst fit among those that contain $P$ is the tangent to the $\lambda_1$-coordinate line though $P$, and this is an ellipse. In order to find the equations of these lines of the best and wort fit, let us
write the coordinate transformation between the Cartesian coordinates and the Jacobi elliptic coordinates (see \cite{ArnoldMMM, DR})
$$
x^2=\frac{(\alpha-\lambda_1)(\alpha-\lambda_2)}{\alpha-\beta},\quad y^2=\frac{(\beta-\lambda_1)(\beta-\lambda_2)}{\beta-\alpha}.
$$
Having in mind that the basal coordinate vectors are $\tilde{\mathbf{n}}_{(i)}=\Big(\frac{\partial x}{\partial \lambda_i}, \frac{\partial y}{\partial \lambda_i}\Big)$ for $i=1,2$ one gets
\begin{equation}\label{eq:n1n2}\begin{aligned}
\tilde{\mathbf{n}}_{(1)}&=-\frac{1}{2}\Big(\frac{\alpha-\lambda_2}{x_P(\alpha-\beta)}\ , \frac{\beta-\lambda_2}{y_P(\beta-\alpha)}\Big), \\ \tilde{\mathbf{n}}_{(2)}&=-\frac{1}{2}\Big(\frac{\alpha-\lambda_1}{x_P(\alpha-\beta)}\ , \frac{\beta-\lambda_1}{y_P(\beta-\alpha)}\Big),
\end{aligned}
\end{equation}
where it is supposed that $\lambda_1$  and $\lambda_2$ are functions of $x_P$ and $y_P$ obtained from \eqref{eq:kvadratna}.
 The pair $(\tilde{\mathbf{n}}_{(1)}, \tilde{\mathbf{n}}_{(2)})$ from \eqref{eq:n1n2} is the solution of the Restricted PCA at the point $P$.

One finally gets the equations of the line that is the best fit:

\begin{equation}\label{worstfit2dline}
\frac{\alpha-\lambda_2}{x_P}(x-x_P)- \frac{\beta-\lambda_2}{y_P}(y-y_P)=0.
\end{equation}
The equations of the line that is the worst fit is obtained analogously:
\begin{equation}\label{bestfit2dline}
\frac{\alpha-\lambda_1}{x_P}(x-x_P)- \frac{\beta-\lambda_1}{y_P}(y-y_P)=0.
\end{equation}

\end{exm}
\begin{exm}\label{ex:cells} Two types of cells in a fraction of the spleens, see \cite{Full}, \cite{Full0}, \cite{CD}, \cite{CDE}.
The data in this example are the numbers of two types of
cells in a specified fraction of the spleens of fetal mice. They are displayed in Table \ref{tab:cells}.  On the basis of sampling, it is reasonable
to assume the original counts to be Poisson random variables, as explained in \cite{CD}. Therefore,
the square roots of the counts are given in the last two columns of the table and they
have, approximately, constant variance equal to $T^{-1}=1/4$. The postulated
model is $y_t = \beta_0 + \beta_1x_t,$
where $(Y_t, X_t) = (y_t, x_t) + (e_t, u_t)$. Here $Y_t$ is the square root of the number of cells
forming rosettes for the $t$-th individual, and $X_t$ is the square root of the number
of nucleated cells for the $t$-th individual. On the basis of the sampling,
the pair of errors $(e_t, u_t)$ have a covariance matrix that is, approximately, $\Sigma= T^{-1}E=\diag(0.25,0.25)$, with $T=4$.
The square roots of the counts cannot be exactly normally distributed, but
we assume the distribution is close enough to normal to permit the use of
the formulas based on normality. Thus, here we assume $\eta=1$.

\begin{table}[h!]
  \begin{center}
    \caption{Numbers of two types of cells; following Cohen and D'Eustachio (1978), see \cite{Full}}
    \label{tab:cells}
    \begin{tabular}{c|c|c|c|c}
    \hline
      j&$m_j$&$n_j$&$Y_j$ & $X_j$  \\ 
      \hline
      1 & 52& 337&7.211 & 18.358 \\ 
      2 & 6 & 141 &  2.449 & 11.874 \\ 
     3 & 14& 177& 3.742 & 13.304 \\
4 & 5& 116 & 2.236 & 10.770  \\
5 & 5& 88&  2.236 & 9.381  \\
     \hline
    \end{tabular}
  \end{center}
\end{table}

The coordinates of the centroid $C$ are $$(\bar x, \bar y)=(12.7374, 3.5748).$$ One calculates that the components of the hyperplanar inertia operator $J_C$ at the centroid $C$ is: $J_{XX}=47.7937, J_{YY}=18.1021, J_{XY}=28.6318$. The principal hyperplanar moments of inertia $J_1, J_2$ are the eigenvalues of the operator $J_C$. The corresponding eigenvectors are directions of the principal axes. The eigenvalues are $J_1= 0.69605, J_2=65.19978$. The corresponding eigenvectors are $\mathbf{n}_1=(-0.51947, 0.85449)^T$, $\mathbf{n}_2=(-0.85449,-0.51947)^T$. The equation of the line that best fits $u_C$ contains the centroid $C$ and is given by
$
u_C: y=0.60793x-4.16865.
$
This coincides with the equations obtained in \cite{Full}. The equation of the line of the worst fit is
$
y=-1.64493x+24.52689.
$

In the original study a hypothesis of interest  is that $\beta_0=0$.
 This corresponds to the condition the line of regression contains the origin of coordinate system. Thus we will consider the origin of coordinate system, denoted by $(X,Y)$, as a point $P$. We want to apply  Theorem \ref{th:GenP1} for a given point $P$.  We introduce $(\tilde{X}, \tilde{Y})$ as the principal coordinates having the centroid $C$ as the origin. Using a coordinate transformation, we recalculate the coordinates of the point $P$ in the principal coordinates and get: $(\tilde{X}_P,\tilde{Y}_P)=(3.56202,12.74098)$.
Using \eqref{eq:2dimpencil}, we get the pencil of conics associated with this data: it is defined with $\alpha=0.13921$, $\beta=-12.76154$:
\begin{equation}\label{eq:confocalcells}
\frac{\tilde{x}^2}{0.13921-\lambda}+\frac{\tilde{y}^2}{-12.76154-\lambda}=1.
\end{equation}

The Jacobi elliptic coordinates of the centroid $C$ are:
$
\lambda_{1_C}=\beta=-12.76154$, $\lambda_{2_C}=\alpha=0.13921.
$
From
$
J_{1}=2J_1-m\lambda_{2_C},
$
we get
$
J_1=m\lambda_{2_C}=m\alpha.
$
Thus the moment of the line $u_C$ is equal to $m\lambda_{2_C}=J_1=0.69605$.

 In order to find the Jacobi elliptic coordinates of the point $P$, associated to the pencil of conics \eqref{eq:confocalcells}, one needs to solve the quadratic equation \eqref{eq:kvadratna}. We get $\lambda_{1_P}=-186.907$ and $\lambda_{2_P}=-0.73589$ and  as the Jacobi elliptic coordinates of the point $P$. Finally, the principal moments of inertia at the point $P$ are
$
J_{1_P}=2J_1-m\lambda_{2_P}=5.071564,\quad J_{2_P}=2J_1-m\lambda_{1_P}=935.9271.
$
Using formulae \eqref{bestfit2dline} and \eqref{worstfit2dline} one gets that the equation of the line $u_P$ that is the best fit among  the lines that contain $P$ is
$
u_P: \tilde{Y}=3.841884\tilde{X}-0.94389.
$
In the original coordinates the equation of this line has the form
$
y=0.30014x.
$
Thus,
$
\hat\beta_1=0.30014.
$
This quantity can be calculated also using \cite{Full}, the equation at the bottom of p. 43, with $M_{xx}=171.8001$, $M_{xy}=51.26003$, $\hat\lambda=4.057252$, $\sigma_{uu}=0.25$, reads as
$
\hat\beta_1=(M_{xx}-\hat\lambda\sigma_{uu})^{-1}M_{xy}=0.30014.
$

The line of best fit $u_P$ is tangent to the hyperbola from the confocal pencil of conics \eqref{eq:confocalcells} which contains $P$:
$$
\frac{\tilde{X}^2}{0.8751}+\frac{\tilde{Y}^2}{-12.0256}=1.
$$
The moment of the line $u_P$ is $m(2\lambda_{2_C}-\lambda_{2_P})=J_{1_P}=5.071564$, with $m=5$.

The line of the worst fit among the lines that contain $P$ is
$
\tilde{Y}=-0.26029\tilde{X}+13.66814.
$
 In the original coordinates the equation of this line has the form
$
y=-3.331376x.
$
It is tangent to the ellipse from the confocal family of conics \eqref{eq:confocalcells} that contains $P$:
$$
\frac{\tilde{X}^2}{187.0462}+\frac{\tilde{Y}^2}{174.1455}=1.
$$
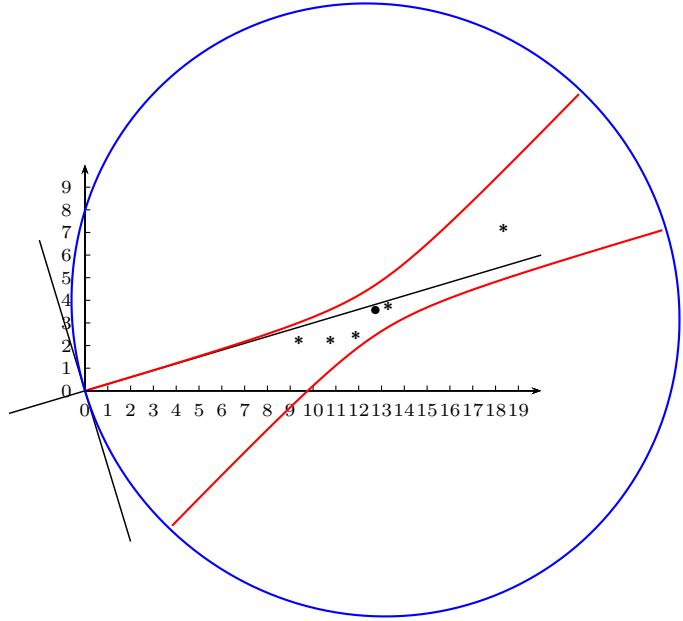
\begin{figure}[h] \centering
\psset{unit =0.7cm, linewidth=0.7\pslinewidth}
\begin{pspicture}(2,-4)(10,7)\psset{xunit=.3cm, yunit=.3cm}
\psaxes[labelFontSize=\scriptstyle, ticksize=2pt]{->}(0,0)(20,10)
\psdots[dotstyle=Basterisk](18.358,7.211)(11.874,2.449)(13.304,3.742)(10.770,2.236)(9.381,2.236)
\psdot[dotstyle=*](12.7374,3.5748)
\psline(-3.33,-1)(20,6)
\psline(-2,6.6626)(2,-6.6626)
\psparametricplot[plotpoints=200, linecolor=blue, linewidth=0.8pt, fillstyle=none]{-180}{180}{12.7374 7.10347 neg t cos  mul add 11.27632 neg t sin mul add 3.5748 11.6864 t cos mul add 6.8542 neg t sin mul add}
\psparametricplot[plotpoints=200, linecolor=red, linewidth=0.8pt, fillstyle=none]{-2}{2}{12.7374 0.4859 neg t COSH mul add 2.9632 neg t SINH mul add  3.5748 0.799348 t COSH mul add 1.8012 neg t SINH mul add}
\psparametricplot[plotpoints=200, linecolor=red, linewidth=0.8pt, fillstyle=none]{-2}{2}{12.7374 0.4859 t COSH mul add 2.9632 neg t SINH mul add  3.5748 0.799348 neg t COSH mul add 1.8012 neg t SINH mul add}
\end{pspicture}
\caption{Example \ref{ex:cells}: The lines of the best and the worst fit tangent to the hyperbola and ellipse respectively, the confocal conics through the point $P$, the origin.}
\label{fig:restortreg}
\end{figure}

See Fig. \ref{fig:restortreg}. The bold point is the center of masses $C$. The two conics from the confocal pencil which pass through the point $P$, the origin,
are presented. The line of the restricted orthogonal regression through the point $P$ is tangent to the hyperbola, while the line of the worst
fit through the point $P$ is tangent to the ellipse.

For testing the null hypothesis $\beta_0 = 0$, the test statistic is \cite{Full}
$
N\hat\lambda/(N-1),
$
whose null distribution can be approximated by Snedecor's F distribution with degrees of freedom $(N-1)$ and $\infty$. Here, as in \cite{Full}, $\hat\lambda$ is the smallest root of the equation $\det(J_P-\lambda\Sigma)=0$, where, as above $\Sigma= T^{-1}E=\diag(0.25,0.25)$, with $T=4$.

 This statistic can be expressed as
$
5\hat\lambda/4=J_{1_P}=5(2\lambda_{2_C}-\lambda_{2_P}),
$
See Theorem \ref{th:Flambda} for a more general statement.
For these data, the value of $J_{1_P}$ is 5.071564. Therefore, the approximate p-value is $P(F_{4, \infty} > 5.07) = 0.00043$, which leads to rejection of the null hypothesis.
\end{exm}
\begin{remark}\label{rem:errorcov} The models with errors in variables for $\eta$ general were first studied by C.H. Kummell in 1879, see \cite{Full}.
Now we want to treat the general cases of models with errors in variables in arbitrary dimension and with a nontrivial error covariance matrix, here denoted as $G$. We assume $G$ is known and positive definite.  As it is well known, see e.g. \cite{Full}, in such cases, the orthogonal regression should be performed not with respect to the ordinary Euclidean metric, but with respect to the metric generated by $G$. Thus, we  consider a  more general case when the inertia operator is defined with respect to the metric $\langle\cdot,\cdot\rangle_G=\langle G\cdot,\cdot\rangle$, where
$\langle\cdot,\cdot\rangle$ is the Euclidean metric and $G$ is a positive definite matrix. The inertia operator is defined by
\begin{equation}\label{metrikaG}
\begin{aligned}
(J^G_O \mathbf{n}_G, \mathbf{m}_G)&=\sum_{i=1}^Nm_i\langle \mathbf{r}_i, \mathbf{n}_G\rangle_G \langle \mathbf{r}_i, \mathbf{m}_G\rangle_G\\
&=\sum_{i=1}^Nm_i\langle \sqrt{G}\mathbf{r}_i, \sqrt{G}\mathbf{n}_G\rangle \langle \sqrt{G}\mathbf{r}_i, \sqrt{G}\mathbf{m}_G\rangle.
\end{aligned}
\end{equation}
 We suppose that $\mathbf{n}_G$ is a unit vector in the metric defined with $G$. Let us introduce $\mathbf{n}=\sqrt{G}\mathbf{n}_G$, $\mathbf{m}=\sqrt{G}\mathbf{m}_G$. Then, $\mathbf{n}$ is a unit vector in the
standard Euclidean metric
$
1=\langle G\mathbf{n}_G, \mathbf{n}_G\rangle=\langle \sqrt{G}\mathbf{n}_G, \sqrt{G}\mathbf{n}_G\rangle=\langle\mathbf{n}, \mathbf{n}\rangle.
$
Thus, we introduce new coordinates:
$
\mathbf{R}_i=\sqrt{G}\mathbf{r}_i.
$
We get
$
(J^G_O \mathbf{n}_G, \mathbf{m}_G)=(J_O\mathbf{n}, \mathbf{m}),
$
where $J_O$ is the inertia operator in the standard Euclidean metric in the new coordinates.
The last relation can be rewritten as
$$\begin{aligned}
(J_O\mathbf{n}, \mathbf{m})&=(J^G_O \mathbf{n}_G, \mathbf{m}_G)=\langle J^G_O \sqrt{G^{-1}}\mathbf{n}, \sqrt{G^{-1}}\mathbf{m}\rangle\\
&=\langle \sqrt{G^{-1}}J^G_O \sqrt{G^{-1}}\mathbf{n}, \mathbf{m}\rangle.\end{aligned}
$$
Thus, we conclude that
$
J_O=\sqrt{G^{-1}}J^G_O \sqrt{G^{-1}}.
$
 Also, if $\mathbf{n}_G$ is an eigen-vector of the operator $J_O^G$, then $\mathbf{n}=\sqrt{G}\mathbf{n}_G$ is an eigen-vector of the operator $J_O$.
The eigenvalue problem \eqref{sistemlambda} in the metric $G$ case can be expressed as
\begin{equation}\label{eq:errorG}
\det({J_O^G-\mu G})=\det(\sqrt{G^{-1}}J_O^G\sqrt{G^{-1}}-\mu E)=\det(J_O-\mu E)=0.
\end{equation}

\end{remark}
\begin{thm}\label{th:Flambda} Let the system of $N$ points $M_1,...,M_N$ with unit masses be given in $\mathbb{R}^k$, $N\ge k$, with the centroid $C$ and the associated pencil of confocal quadrics \eqref{kdimkonfokal}.  For any point $P(x_{01},...,x_{0k})$ denote its Jacobi coordinates of the pencil of confocal quadrics \eqref{kdimkonfokal} by $(\lambda_{1_P}<\dots<\lambda_{k_P})$ and the Jacobi coordinates of the centroid by $(\lambda_{1_C}<\dots<\lambda_{k_C})$. Then:
\begin{enumerate}[label=(\alph*)]
\item
 The hyperplanar moment of the hyperplane of the best fit is equal to $J_1=N\lambda_{k_C}$.
 \item
 The hyperplanar moment of the hyperplane of the best fit that contains the point $P$ is equal to $J_P=N(2\lambda_{k_C}-\lambda_{k_P})$.
\item The test statistic of the hypothesis that the hyperplane of the best fit contains the point $P$ is:
\begin{equation}\label{eq:Flambda}
\frac{N}{N-k+1}(2\lambda_{k_C} - \lambda_{k_P}).
\end{equation}
whose null distribution can be approximated by Snedecor's F distribution with degrees of freedom $(N-k+1)$ and $\infty$.
\end{enumerate}
\end{thm}

\begin{proof}
The proof follows the lines of the consideration from Example 4.1, with the obvious higher-dimensional adjustments.
First, from the formula \eqref{eq:mul} applied to the centroid $C$ we get:
$$
J_1=2J_1-N\lambda_{k_C},
$$
and deduce
$$
J_1=N\lambda_{k_C}.
$$
From the last formula and \eqref{eq:mul} we get that the hyperplanar moment of the hyperplane of the best fit is
$$
RSS_2=N\lambda_{k_C}.
$$
By using the same formulae, we get the hyperplanar moment of the hyperplane of the best fit containing a given point $P$ to be
$$
RSS_1=N(2\lambda_{k_C}-\lambda_{k_P}).
$$
From Theorem 2.3.2, formula (2.3.19), from \cite{Full} for the approximation of the test statistic we get:
$$
N(N-k-1)^{-1}\tilde\lambda,
$$
where $\tilde\lambda$ is the smallest root of $\det(M-\lambda T^{-1}S)=0$, see formula (2.3.18) in \cite{Full}. Here $T$ is a positive scalar. To relate to our notation, one takes $J_O^G=NM$ and $G=T^{-1}S$.
From there we get that the test statistic is approximated by
$$
\frac{N}{N-k+1}(2\lambda_{k_C} - \lambda_{k_P}),
$$
where $\lambda_{k_C}$ and $\lambda_{k_P}$ are the largest Jacobi coordinates of the points $C$ and $P$, where the Euclidean coordinates of these points are calculated in the coordinates $X=Gx$, and the confocal pencil \eqref{konlambde} (the same as \eqref{kdconfocal}) is generated by $\sqrt{G^{-1}}J_O^G\sqrt{G^{-1}}$.
\end{proof}
\begin{remark} If we take $G_1=S$, then the test statistic is approximated by
$$
\frac{NT}{N-k+1}(2\lambda_{k_C} - \lambda_{k_P}),
$$
where $\lambda_{k_C}$ and $\lambda_{k_P}$ are the largest Jacobi coordinates of the points $C$ and $P$, where the Euclidean coordinates of these points are now calculated in the coordinates $X=G_1x$, and the confocal pencil  is generated by $\sqrt{G_1^{-1}}J_O^G\sqrt{G_1^{-1}}$.
\end{remark}

\section{The directional regressions and ellipsoids}\label{sec:dirregel}
In the previous sections we dealt primarily with geometric aspects of orthogonal regression.
From the work of Galton \cite{Gal} we know that conics play significant role in simple linear regression.
From Fig. \ref{fig:Galton} one observes that points of horizontal and vertical tangency to the Galton ellipse (aka the data ellipse)
 determine the lines of linear regression of $x$ on $y$ and of $y$ on $x$ respectively. This motivates our further study of linear regression in $\mathbb R^k$ in a geometric, invariant, i.e. coordinate free way. We consider linear regression in an arbitrary selected direction, not necessarily horizontal or vertical. In Section \ref{sec:dir2} we deal with the planar case and generalize this to an arbitrary $k$ in Section \ref{sec:dirk}.
A higher-dimensional geometric generalization of the above observation about Galton ellipses is given in Theorem \ref{th:dirregcondk}, see also Corollary \ref{cor:k2dirreg}.
The family of confocal quadrics \eqref{kdimkonfokal} which we associated with the given system of points comes into picture here in reaching our third goal: {\it for a given direction, for a given system of $N$ points, under the  full rank assumption, and for a selected point, find the best fitting hyperplane in  the given direction}. This is done in Theorem \ref{thm8.4}. The effectiveness of the developed methods and results and their applications in statistics are illustrated in Examples \ref{ex:2ddirectional} and \ref{ex:pressure}.

\subsection{The directional axial moments in the plane}\label{sec:dir2}

Given a  line $w$ in the plane. Given $N$  points $M_1,...,M_N$ with masses $m_1, m_2,..., m_N$ of the total mass $m$ in the plane.  The
 point $C$ denotes the center of masses.

For a given line $u\subset \mathbb R^2$, non-parallel to $w$, \emph{the axial moment of inertia in direction $w$}, denoted by $I^w_u$, is defined as
$$I^w_u=\sum\limits_{j=1}^{N}m_j\hat d_j^2,$$
where $\hat d_j$ is the distance from the point $M_j$ to the point of intersection of the line $u$ with the line parallel to $w$ through the point $M_j$, $j=1, \dots, N$.

Suppose the line $u$ contains a point $O$. Denote by $\textbf{u}_0$ the unit vector parallel to the line $u$ and by $\textbf{w}_0$ the unit vector parallel to the line $w$. As before,  see Section \ref{sec:planaraxial}, we denote by $I_u$ the axial moment of inertia for the axis $u$ and by $I_O$ the operator of inertia for the point $O$. Then the axial moment of inertia $I^w_u$ for the axis $u$ in direction $w$ can be rewritten in the form:

\begin{equation}\label{eq:axialdir}
I^w_u=\sum\limits_{i=1}^{N}m_i\hat d_i^2=I_u\cdot \frac{1}{1-\langle\textbf{u}_0,\textbf{w}_0\rangle^2}=\frac{\langle I_O\textbf{u}_0,\textbf{u}_0\rangle}{1-\langle\textbf{u}_0,\textbf{w}_0\rangle^2}.
\end{equation}

\begin{figure}[h]
\begin{minipage}{0.49\textwidth}
\includegraphics[width=5.3cm,height=3cm]{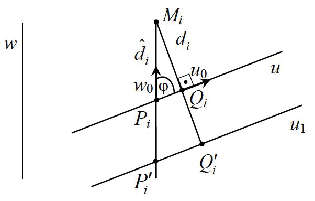}
\caption{With formula \eqref{eq:axialdir} and with Proposition \ref{prop:dirHST}.}
\label{fig:dirHST}
\end{minipage}
\begin{minipage}{0.49\textwidth}
\includegraphics[width=5.3cm,height=3cm]{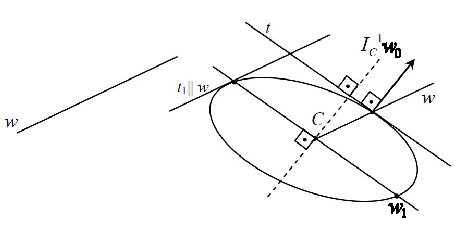}
\caption{With Theorem \ref{th:axialgyration}.}
\label{fig:axialgyration}
\end{minipage}
\end{figure}

From the last formula and Fig. \ref{fig:dirHST}, one easily gets, the following, directional version of the Huygens-Steiner Theorem:

\begin{prop}\label{prop:dirHST}[The directional Huygens-Steiner Theorem]
 Let the axis $u$ contain the center of masses $C$ and let $u_1$ be a line parallel to $u$. Denote by $I^w_{u_1}$ and $I^w_u$ the corresponding
 directional axial moments of inertia of a given system of points with the total mass $m$ in the direction $w$. Then

\begin{equation}\label{eq:HSdir} I^w_{u_1}=I^w_u+m\hat d^2,
\end{equation}
where $\hat d$ is the distance between the points of intersection of a line parallel to $w$ with the parallel lines $u$ and $u_1$.
\end{prop}

Thus, again we get a characterization of the center of masses, as a consequence of \eqref{eq:HSdir}:

\begin{cor} Given a direction $w$, the system of points and a direction $u$ not parallel to $w$.
Among all the lines parallel to $u$, the least directional moment of inertia in direction $w$ is attained by the line through the center of masses of the system of points.
\end{cor}

We will investigate the lines which minimize the directional axial moment in the direction $w$ and call them the directional lines of regression in the direction $w$. From the above Corollary we see that they have to pass through the center of masses $C$. Define the central axial ellipse
of gyration as
\begin{equation}\label{eq:gyrellipse}
\langle I^{-1}_C\textbf{x},\textbf{x}\rangle=1.
\end{equation}
This is the same as the data ellipse.

Denote by $\hat u_w$ the direction of the directional line of regression in the direction $w$.

\begin{thm}\label{th:axialgyration} Let $w$ and $w_1$ form a pair of conjugate directions of the axial ellipse of gyration \eqref{eq:gyrellipse}.
Then the direction of the directional line of regression in the direction $w$ is orthogonal to $w_1$, i.e.
$
\langle \hat u_w,w_1\rangle=0.
$
\end{thm}

\begin{proof}
See Fig. 8.
Let us use the formula \eqref{eq:axialdir} and calculate the gradient of $I^w_u$ with respect to $u$.
We get
$$
\grad_{\textbf{u}_0}I^w_u=\alpha I_C\textbf{u}_0 +\beta \textbf{w}_0,
$$
where $\alpha, \beta$ are some scalars and $\alpha\ne 0$.  Since $\mathbf{u}_0$ is a unit vector, or in other words, $g(\mathbf{u}_0)=\langle\mathbf{u}_0,\mathbf{u}_0\rangle-1=0$, in order to find the directional line of regression, we need to find the conditional extremum. Let us denote by $\mu$ the Lagrange multiplier.
Thus,
$$
\grad_{\textbf{u}_0}I^w_u=\mu \grad_{\textbf{u}_0}{g(\mathbf{u}_0)},
$$

only if the vectors $I_C\textbf{u}$ and $\textbf{w}_0$ are collinear. We get the condition
$
\hat u_w=I_C^{-1}w_0.
$
The direction $w_1$ is conjugate to $w$ with respect to the the axial ellipse of gyration \eqref{eq:gyrellipse} if and only if
$
\langle I_C^{-1}w_0,w_1\rangle=0.
$
This proves the theorem.
\end{proof}

 Now we consider a similar problem, but searching for the best fit among the lines that contain a given point $P$, distinct from the center of masses: \emph{Find the line in the plane that contains a given point $P$
and has the minimal directional axial moment in the direction $w$}.

Let $u$ and $u_1$ be two parallel lines, that contains $C$ and $P$ respectively. We assume that $u$ and $u_1$ are not parallel to $w$. Applying the Huygens-Steiner theorem \eqref{eq:HSdir}, we get
$
I^w_{u_1}=I^w_{u}+m\hat{d}^2.
$
Since $\hat{d}^2=\frac{{d}^2}{1-\langle u_0,w_0\rangle^2}$  using \eqref{eq:axialdir}, we get
$
I_{u_1}^w=\frac{I_u+m{d}^2}{1-\langle u_0,w_0\rangle^2}=\frac{I_{u_1}}{1-\langle u_0,w_0\rangle^2}=\frac{\langle I_Pu_0,u_0\rangle}{1-\langle u_0,w_0\rangle^2}.
$
The last formula is an analog of \eqref{eq:axialdir}. Thus, we get:
\begin{thm}\label{th:dirregcond} Let $w$ and $w_1$ form a pair of conjugate directions of the axial ellipse of gyration $\langle I^{-1}_Px,x\rangle=1$ at a given point $P$.
Then the direction of the directional line of regression in the direction $w$ at the point $P$ is orthogonal to $w_1$, i.e.
$
\langle \hat u_w,w_1\rangle=0.
$
\end{thm}

\subsection{The directional hyperplanar moments in $\mathbb R^k$}\label{sec:dirk}

Let us select a  line $w$ in $\mathbb R^k$. Given $N$  points $M_1,...,M_N$ with masses $m_1, m_2,..., m_N$ of the total mass $m$ in $\mathbb R^k$.  The
 point $C$ denotes the center of masses.

Now, we introduce  the hyperplanar moment of inertia for the given system of points and for a given hyperplane $\pi\subset \mathbb R^k$, which is not parallel to $w$. We define $J^w_{\pi}$
the hyperplanar moment of inertia in direction $w$ as follows:
$J^w_{\pi}=\sum\limits_{j=1}^Nm_j\hat D_j^2,$
where $\hat D_j$ is the distance between the point $M_j$ and the intersection of the hyperplane $\pi$ with the line parallel to $w$ through $M_j$, for $j=1, \dots, N$. Let  $\textbf{n}$ be the unit vector orthogonal to the hyperplane $\pi$, and $O$ a point contained  in $\pi$. Then, using Fig. \ref{fig:hyperplanardir} the hyperplanar moment of inertia $J^w_{\pi}$ can be rewritten in the form
\begin{equation}\label{eq:pmidirect}
J^w_{\pi}=\frac{J_{\pi}}{\langle\textbf{w}_0,\textbf{n}\rangle^2}=\frac{\langle J_O\textbf{n},\textbf{n}\rangle}{\langle\textbf{w}_0,\textbf{n}\rangle^2}.
\end{equation}
Here  $J_{\pi}$ is
the hyperplanar moment of inertia and the operator $J_O$ is the hyperplanar inertia operator.

\begin{prop}[The hyperplanar directional Huygens-Steiner Theorem]
 Let the hyperplane $\pi$ contain the center of masses $C$ and let $\pi_1$ be a hyperplane parallel to $\pi$. Denote by $J^w_{\pi_1}$ and $J^w_{\pi}$ the corresponding
 directional hyperplanar moments of inertia of a given system of points with the total mass $m$ in the direction $w$. Then

\begin{equation}\label{eq:HSdirplan} J^w_{\pi_1}=J^w_{\pi}+m\hat D^2,
\end{equation}
where $\hat D$ is the distance between the points of intersection of a line parallel to $w$ with the parallel hyperplanes $\pi$ and $\pi_1$.
\end{prop}

Thus, again we are getting a characterization of the center of masses, using \eqref{eq:HSdirplan}:

\begin{cor} Given a direction $w$, the system of points and one hyperplane $\pi$ not parallel to $w$.
Among all the hyperplanes parallel to $\pi$, the least directional hyperplanar moment of inertia in direction $w$ is attained for the hyperplane which contains the center of masses of the system of points.
\end{cor}

\begin{figure}[h]
\begin{minipage}{0.49\textwidth}\centering
\includegraphics[width=5.3cm,height=3cm]{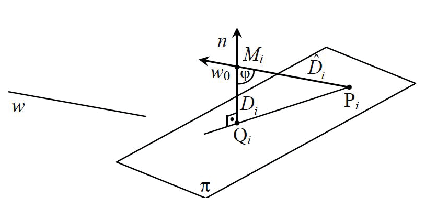}
\caption{With equation \eqref{eq:pmidirect}.}
\label{fig:hyperplanardir}
\end{minipage}
\begin{minipage}{0.49\textwidth}\centering
\includegraphics[width=5.3cm,height=3cm]{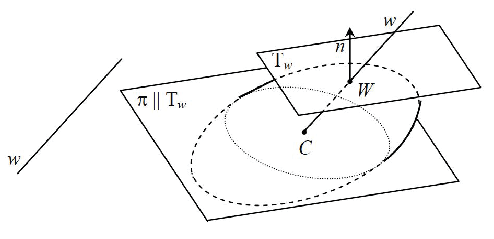}
\caption{With Theorem \ref{th:dirregcondk}.}
\label{fig:dirregcondk}
\end{minipage}
\end{figure}

We will investigate the hyperplanes which minimize the directional hyperplanar moment in the direction $w$ and call them the directional hyperplanes of regression in the direction $w$. From the above Corollary it follows that they contain the center of masses $C$. Define the data ellipsoid as
\begin{equation}\label{eq:concentellipsoid}
\langle J^{-1}_C\textbf{x},\textbf{x}\rangle=1.
\end{equation}
Denote by $\hat \pi_w$  the directional hyperplane of regression in the direction $w$.

\begin{thm}\label{th:dirregcondk} Let the radius vector from the center of masses $C$ in direction $w$ intersect the data ellipsoid \eqref{eq:concentellipsoid} at the point $W$ and let the tangent hyperplane to the data ellipsoid at $W$ be $\pi_W$.
The directional hyperplane of regression in the direction $w$ contains $C$ and is parallel to the tangent hyperplane $\pi_W$.
\end{thm}

\begin{proof}
Let us use the formula \eqref{eq:pmidirect}  and calculate the gradient of $J^w_{\pi}$ with respect to the normal vector $n$.
We get
$$
\grad_{\textbf{n}}J^w_{\pi}=\alpha J_C\textbf{n} +\beta \textbf{w}_0,
$$
where $\alpha, \beta$ are some scalars and $\alpha\ne 0$.  Since $\mathbf{n}$ is a unit vector, i.e. $g(\mathbf{n})=\langle\mathbf{n},\mathbf{n}\rangle-1=0$, in order to find the directional hyperplane of regression, we need to find the conditional extremum. Let us denote by $\mu$ the Lagrange multiplier.
Thus,
$$
\grad_{\textbf{n}}J^w_{\pi}=\mu\grad_{\textbf{n}}{g(\mathbf{n})},
$$
only if the vectors $J_C\textbf{n}$ and $\textbf{w}_0$ are collinear. We get that
$
\hat n_w=J_C^{-1}\mathbf{w}_0
$
is the vector orthogonal to the directional hyperplane of regression in the direction $w$. On the other hand, $J_C^{-1}\mathbf{w}_0$ is orthogonal to the tangent hyperplane of the data ellipsoid \eqref{eq:concentellipsoid} at its point $W$. See Fig. 10.
This proves the theorem.
\end{proof}

A statement similar to the last theorem can be extracted from \cite{Dem}, based on different considerations.
For $k=2$ we get an important specialization of Theorem \ref{th:dirregcondk}, see Fig. \ref{fig:k2dirreg}.

\begin{cor}\label{cor:k2dirreg} In the case $k=2$ the line of regression in the direction $w$ contains the center of masses $C$ and intersects the data ellipse  at two points at which the tangents to the data ellipse  are parallel to the line $w$. Equivalently, the line of regression is parallel to the tangent line of the data ellipse  at the point of its intersection with the line parallel to $w$ passing through the centroid $C$.
\end{cor}

\begin{figure}[h]
\begin{minipage}{0.49\textwidth}\centering
\includegraphics[width=5.3cm,height=3cm]{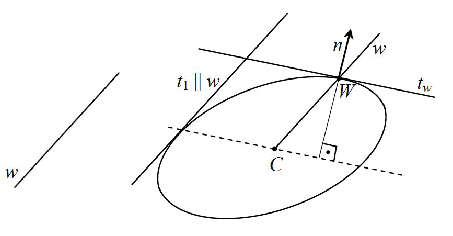}
\caption{With Corollary \ref{cor:k2dirreg}.}
\label{fig:k2dirreg}
\end{minipage}
\begin{minipage}{0.49\textwidth}\centering
\includegraphics[width=5.3cm,height=3cm]{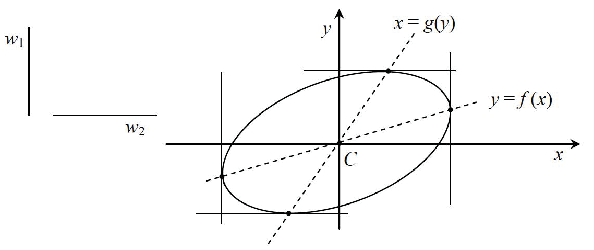}
\caption{Simple linear regression.}
\label{fig:slr}
\end{minipage}
\end{figure}

\begin{exm}\label{ex:slr} Let us consider the two-dimensional case of the last Corollary in more detail. Such a situation is typical in statistics in simple linear models.
Assume that each observation from an experiment is a pair of numbers, and that the goal is
 to predict one of the numbers from the other. Then the method of least squares can be
used for constructing a predictor of one of the variables from the other one, by
 using of a sample of observed pairs (see e.g. \cite{DeSch}).

 Given $N$ points $(x^{(i)}, y^{(i)}),\,_{i=1}^N$ in the plane. The straight line
 $$y=f(x)$$
 that minimizes the sum of the squares of the vertical deviations of all the points from the
line is called the least-squares line in the direction of the $y$-axis. Similarly, the straight line
$$x=g(y)$$
that
minimizes the sum of the squares of the horizontal deviations of all the points from the
line is called the least-squares line in the direction of the $x$-axis. See Fig. \ref{fig:slr}, where the direction $w_1$ is parallel
to the $y$-axis, while $w_2$ is parallel to the $x$-axis. The  least-squares line $y=f(x)$ in the direction of the $y$-axis
intersects the data ellipse at the points of vertical tangency. The  least-squares line $x=g(y)$ in the direction of the $x$-axis
intersects the data ellipse at the points of horizontal tangency.
If
\begin{equation}\label{eq:f}
f(x)= \hat\beta_0 +\hat\beta_1 x
\end{equation}
then, it is well known (see e.g. \cite{DeSch}), that the pair of coefficients $(\hat\beta_0, \hat\beta_1)$ can be determined as follows
\begin{equation}\label{eq:directbeta}
\hat \beta_1=\frac{\sum_{i=1}^N(y^{(i)}-\bar y)(x^{(i)}-\bar x)}{\sum_{i=1}^N(x^{(i)}-\bar x)^2}, \quad \hat\beta_0=\bar y -\hat \beta_1\bar x,
\end{equation}
where $(\bar x, \bar y)$ are the coordinates of the centroid of the given system of $N$ points:
$$
\bar x=\frac{1}{N}\sum_{i=1}^Nx^{(i)}, \quad \bar y=\frac{1}{N}\sum_{i=1}^Ny^{(i)}.
$$
By exchanging the roles of $x$ and $y$ in the above formulas we get the coefficients $\hat \alpha_0, \hat\alpha_1$ of the linear function $g(y)=\hat\alpha_0 + \hat\alpha_1 y$.

If we assume that the $y$ values are observed values of a
collection of random variables $Y$, then under the assumptions of the simple linear regression,  the
method of least squares provides the maximum likelihood estimates
$\hat\beta_0$ and $\hat\beta_1$ for the parameters of the model $\beta_0$ and $\beta_1$ respectively.
Moreover
$$
\hat\sigma^2=\frac{1}{N}\sum_{i=1}^N(y^{(i)}-\hat\beta_0-\hat\beta_1x^{(i)})^2
$$
is the  maximum likelihood estimate for the variance $\sigma^2$ of the normal distribution describing the conditional distribution of $Y$ given $X=x$. While $\hat\beta_0$ and $\hat\beta_1$ are unbiased estimators, $\hat\sigma^2$ is a biased estimator. The corresponding unbiased estimator
is
$$
\hat\sigma'^2=\frac{S^2}{N-2},
$$
where
$$
S^2=\sum_{i=1}^N(Y^{(i)}-\hat\beta_0-\hat\beta_1x^{(i)})^2.
$$

Now, we will derive the formulas \eqref{eq:f}, \eqref{eq:directbeta} directly from Theorem \ref{th:dirregcondk} and Corollary \ref{cor:k2dirreg}.  Let us present the formula for the line that minimize
sum of the squares of deviations in the direction $\mathbf{w}_0=(w_x,w_y)^T$. As mentioned before, this line contains the
centroid (i.e. the center of masses) $(\bar{x},\bar{y})$ of the system of points. Let its equation be of the form
$$y-\bar{y}=\kappa(x-\bar{x}).$$
From Theorem \ref{th:dirregcondk} it follows that this line is parallel to
$\hat n_w=J_C^{-1}\mathbf{w}_0$. Since
$$
J_C=\left(\begin{matrix}J_{11}&J_{12}\\
                        J_{12}&J_{22}
                        \end{matrix}
                        \right)
$$
one calculates
$$
J_C^{-1}=\frac{1}{\det J_C}\left(\begin{matrix}J_{22}&-J_{12}\\
                        -J_{12}&J_{11}
                        \end{matrix}
                        \right).
$$
Thus,
$$
\hat n_w=\frac{1}{\det J_C}\left(\begin{matrix}J_{22}w_x-J_{12}w_y\\
                                                -J_{12}w_x+J_{11}w_y
                                                \end{matrix}\right),
$$
and one gets that the slope of the line is
$$\kappa=\frac{J_{22}w_x-J_{12}w_y}{J_{12}w_x-J_{11}w_y}.$$
In the special case of the vertical deviations, $\mathbf{w}_0=(0,1)$, and the equation of the least squares line becomes
$$
y-\bar{y}=\frac{J_{12}}{J_{11}}(x-\bar{x}).
$$
Using
$$
J_{12}=\sum\limits_{i=1}^N(y^{(i)}-\bar{y})(x^{(i)}-\bar{x}),\quad J_{11}=\sum\limits_{i=1}^N(x^{(i)}-\bar{x})^2,
$$
we get both formulas \eqref{eq:f} and \eqref{eq:directbeta}.
\end{exm}

\begin{exm}
In a similar manner as in Example \ref{ex:slr}, one can consider the general linear model and multiple regression. One can derive the equations in $\mathbb R^k$ of the hyperplane that minimizes the sum of the squares of deviations in the direction $\mathbf{w}_0=(w_1,w_2,...,w_k)^T$. The normal vector
$\mathbf{n}=J_C^{-1}\mathbf{w}_0$ and the equation of the plane is
$$
\sum\limits_{i=1}^N n_i(x_i-\bar{x}_i)=0.
$$
 For example, in the direction normal to hyperplane $x_i=1$, the vector of the normal to the plane $\mathbf{n}$ is the $i$-th column of the matrix of cofactors of $J_C$.

A standard derivation of formulas of the general linear model and multiple regression in direction $x_k$ one can find in e.g. \cite{DeSch}, Ch. 11.5.
\end{exm}

Let the direction $w$ and a system of $N$ points in $\mathbb{R}^k$ with masses $m_1,...,m_N$ be given under the full rank assumption, where $k\geqslant 2$.  We are going to describe all hyperplanes that have the same directional hyperplanar moment of inertia in direction $w$. We consider a coordinate system  with the origin at the centroid $C$, and with the direction $w$ coinciding with the axis $Cx_{k}$.

The  hyperplanar operator of inertia at the point $C$ has the expression
$$
J_{C}=\left( \begin{matrix}
J_{11}&J_{12}&J_{13}&...&J_{1k}\\
J_{12}&J_{22}&J_{23}&...&J_{2k}\\
...&&&&&\\
J_{1k-1}&J_{2k-1}&J_{3k-1}&...&J_{k-1k}\\
J_{1k}&J_{2k}&J_{3k}&...&J_{kk}\\
\end{matrix}
\right).
$$
It is convenient to introduce the submatrix $K_1$ of the matrix $J_C$ and the vector $b$ by:
$$
K_1=\left( \begin{matrix}
J_{11}&J_{12}&J_{13}&...&J_{1k-1}\\
J_{12}&J_{22}&J_{23}&...&J_{2k-1}\\
...&&&&&\\
J_{1k-1}&J_{2k-1}&J_{3k-1}&...&J_{k-1k-1}\\
\end{matrix}
\right),
$$
$$
b=(J_{1k},...,J_{k-1k})^T.
$$
In the case when all masses are equal to $m_i=\frac{1}{N}$, the matrix $K_1$ coincides with the covariance matrix $K$.
Assume that the hyperplane $\pi$ is not parallel to the axis $Cx_k$. The equation of $\pi$ can be  given in the form
\begin{equation}\label{planepi}
x_k=\beta_0-\beta_1x_1-\beta_2x_2-...-\beta_{k-1}x_{k-1}.
\end{equation}
Let us denote $\beta=(\beta_1,...,\beta_{k-1})$.
\begin{prop}\label{prop:directhyper}
Given a real number $\mu$, let a hyperplane \eqref{planepi} have the directional hyperplanar moment of inertia in direction $Cx_k$ equal to $\mu$. Then the coefficients $\beta_0, \beta_1,...,\beta_{k-1}$ of the hyperplane \eqref{planepi} satisfy the equation
\begin{equation}\label{coefeq}
\langle K_1\beta,\beta\rangle-2\langle b,\beta\rangle+m\beta_0^2+J_{kk}=\mu.
\end{equation}
By varying the moment of inertia $\mu$, the equations \eqref{coefeq} represent the family of homothetic ellipsoids in the space of parameters $\beta_0, \beta_1,...,\beta_{k-1}$.
\end{prop}
\begin{proof}
We calculate the directional hyperplanar moment of inertia of the hyperplane \eqref{planepi}:
\begin{equation}\label{dirconst}
\begin{aligned}
J_{\pi}^{x_k}&=\sum_{i=1}^Nm_id_i^2=\sum_{i=1}^Nm_i\big(x^{(i)}_k-\beta_0-\beta_1x^{(i)}_1-...-\beta_{k-1}x^{(i)}_{k-1}\big)^2\\
&=J_{kk}+m\beta_0^2+\beta_1^2J_{11}^2+...+\beta_{k-1}^2J_{k-1\ k-1}-2\sum_{i=1}^{k-1}\beta_iJ_{ik}+2\sum_{i,j=1\atop i<j}^{k-1}\beta_i\beta_jJ_{ij}\\
&=\langle K_1\beta,\beta\rangle-2\langle b,\beta\rangle+m\beta_0^2+J_{kk}.
\end{aligned}
\end{equation}
This gives a proof of the first part of the Proposition. We used here that $C$ is the centroid and consequently, $\sum_{i=1}^N m_ix_j^{(i)}=0$ for $j=1,...,k$.

In the space of parameters $\beta_0, \beta_1,...,\beta_{k-1}$, the equation \eqref{coefeq} defines a family of quadrics that depend on $\mu$.
Since $K_1$ is a positive definite matrix, this is a family of ellipsoids.

\end{proof}

\begin{prop}\label{prop:homotetic} All ellipsoids from the family \eqref{coefeq} have the same center at the point $(0,\hat{\beta})$, where $\hat{\beta}=K_1^{-1}b$ are the coordinates of the best fit hyperplane. The formula \eqref{coefeq} can be rewritten in the form
\begin{equation}\label{coefeq1}
\langle K_1(\beta-\hat{\beta}),\beta-\hat{\beta}\rangle=\mu+\langle K_1^{-1}b,b\rangle-m\beta_0^2-J_{kk}.
\end{equation}

\end{prop}
\begin{proof} We have
$$
\begin{aligned}
\mu&=\langle K_1\beta,\beta\rangle-2\langle b,\beta\rangle+m\beta_0^2+J_{kk}\\
&=\langle K_1(\beta-K_1^{-1}b,\beta-K_1^{-1}b\rangle-\langle K_1^{-1}b,b\rangle+m\beta_0^2+J_{kk}.
\end{aligned}
$$
By introducing $\hat{\beta}=K_1^{1}b$, one gets the formula \eqref{coefeq1}. The coordinates of the vector $\hat{\beta}$ give the best fit hyperplane obtained by least square method (see for example formula (3.6) from \cite{HTF}).

\end{proof}

Proposition \ref{prop:homotetic} shows that the intersection of the family \eqref{coefeq} with the hyperplane $\beta_0=0$ coincides with the family of homothetic ellipsoids of residuals \eqref{eq:residual}.

One can always choose coordinates in which $K_1$ is diagonal, and thus the formula \eqref{coefeq} takes a simpler form. The direction $w$ coincides with $Cx_k$. In the chosen coordinates all the centrifugal moments of inertia $J_{ij}$ for $i,j=1,...,k-1$ are equal to zero. In other words, the  hyperplanar operator of inertia at the point $C$ has the expression
$$
J_{C}=\left( \begin{matrix}
J_1&0&0&...&0&J_{1k}\\
0&J_2&0&...&0&J_{2k}\\
...&&&&&\\
0&0&0&...&J_{k-1}&J_{k-1k}\\
J_{1k}&J_{2k}&J_{3k}&...&J_{k-1k}&J_{k}\\
\end{matrix}
\right).
$$

Let us fix $J_{\pi}^{x_k}=\mu$. Then using that
$$
\det(J_C)=J_1J_2...J_{k-1}\big( J_k-\frac{J_{1k}}{J_1}-...-\frac{J_{k-1k}}{J_{k-1}}\big)
$$
one gets that the equation of the family of ellipsoids \eqref{coefeq} has the form

\begin{equation}\label{coefeqdiag}
J_1\Big(\beta_1-\frac{J_{1k}}{J_1}\Big)^2+...+J_{k-1}\Big(\beta_{k-1}-\frac{J_{k-1k}}{J_{k-1}}\Big)^2
+m\beta_0^2=\mu-\frac{\det(J_{C})}{J_1J_2...J_{k-1}}.
\end{equation}
In such chosen coordinates, the center of the ellipsoids is $(0, \frac{J_{1k}}{J_1},...,\frac{J_{k-1k}}{J_{k-1}})$.

A geometric approach to the lasso with use of Proposition \ref{prop:directhyper} and Proposition \ref{prop:homotetic} is presented in \cite{DG2024}.

 Now, as in the two-dimensional case, we consider the  following problem  for hyperplanes that contain a given point $P$ distinct from the center of masses: \emph{Find the hyperplane  that contains a given point $P$
and has the minimal directional hyperplanar moment in the direction $w$}.

Let $\pi$ and $\pi_1$ be two parallel hyperplanes that contain $C$ and $P$ respectively. We assume that  $\pi$ and $\pi_1$ are not orthogonal to $w$. The Huygens-Steiner theorem \eqref{eq:HSdirplan} gives:
$
J^w_{\pi_1}=J^w_{\pi}+m\hat{D}^2.
$
Since $\hat{D}^2=\frac{{D}^2}{\langle u_0,w_0\rangle^2}$  using \eqref{eq:pmidirect}, we get
$
J_{\pi_1}^w=\frac{J_\pi+m{D}^2}{\langle u_0,w_0\rangle^2}=\frac{J_{\pi_1}}{\langle u_0,w_0\rangle^2}=\frac{\langle J_Pu_0,u_0\rangle}{\langle u_0,w_0\rangle^2}.
$
Here $J_P$ is given by \eqref{JuP}. The last formula is an analog of \eqref{eq:pmidirect}. In statistics, the data ellipsoid is typically considered at the centroid. In mechanics ellipsoids of inertia are not exclusively considered at the center of masses. Transporting this important flexibility from mechanics to statistics, we consider the data  ellipsoid  at a given point. Let us recall that we described its principal hyperplanes in Theorem
\ref{th:principal1}. Here the confocal pencil of quadrics \eqref{kdimkonfokal} we constructed associated to the given system of points appears again.
The principal hyperplanes are tangent hyperplanes at the point $P$ of those quadrics from the confocal pencil \eqref{kdimkonfokal}, which contain $P$. We get
\begin{thm}\label{thm8.4}
 Given the data ellipsoid  at a point $P$
$
\mathcal E_P:\, \langle J^{-1}_P\textbf{x},\textbf{x}\rangle=1.
$
  Let the radius vector from the point $P$ in direction $w$ intersect the ellipsoid $\mathcal E_P$  at the point $W$. Let the tangent hyperplane to the ellipsoid $\mathcal E_P$ at $W$ be $\pi_W$.
The directional hyperplane of regression in the direction $w$, among the hyperplanes which contain the point $P$,  is parallel to the tangent hyperplane $\pi_W$.  In other words, the directional hyperplane of regression is the hyperplane conjugate to $PW$ with respect to $\mathcal E_P$.

\end{thm}

\begin{exm}\label{ex:2ddirectional}  Let $N$ points with masses $m_1,...,m_N$ be given in the plane. Let us also fix a point $P$ in the plane.
Among the lines that contain the given point $P$, we will derive formulas for the directional line of regression in the direction $w_0=(w_x,w_y)$ for the given system of points in the plane, using the associated confocal pencil of conics \eqref{kdimkonfokal}. We will provide the answer in terms of the Jacobi elliptic coordinates of the point $P$ with respect to the associated confocal pencil of conics.

We use the notation from Example \ref{exam2dim}.  Denote by $C$ the centroid of the system of $N$ points. Denote $Cxy$ the principal coordinate system of the confocal pencil and $J_1$ and $J_2$ as before, denote the  principal planar moments of the inertia. Introduce $a_2$ such that
$J_1+ma_2^2=J_2.$
 The confocal pencil of conics associated to the system of points \eqref{kdimkonfokal} has the form
\begin{equation}\label{eq:pencil2ddirectiona}
\frac{x^2}{\alpha-\lambda}+\frac{y^2}{\beta-\lambda}=1,\quad \alpha=\frac{J_1}{m},\quad \beta=\frac{J_1}{m}-a_2^2.
\end{equation}

We calculated the
principal vectors $\tilde{\mathbf{n}}_{(1)}$ and $\tilde{\mathbf{n}}_{(2)}$ for the planar operator of inertia at the point $P$, see \eqref{eq:n1n2}. In the corresponding orthogonal basis $[\mathbf{n}_{(1)}, \mathbf{n}_{(2)}]$, the planar operator of inertia has a diagonal form
$J_P=\diag(J_{P_1}, J_{P_2}).$ The matrix of the change of bases $S$ is
$$
S=\left(\begin{matrix}
\frac{\alpha-\lambda_2}{\Delta_1 x_P}&\frac{\alpha-\lambda_1}{\Delta_2 x_P}\\
-\frac{\beta-\lambda_2}{\Delta_1 y_P}&-\frac{\beta-\lambda_1}{\Delta_2 y_P},
\end{matrix}\right),\, \text{where}$$
$$
\Delta_1=\sqrt{\frac{(\alpha-\beta)^2(\lambda_2-\lambda_1)}{(\alpha-\lambda_1)(\beta-\lambda_1)}},\
\Delta_2=\sqrt{\frac{(\alpha-\beta)^2(\lambda_2-\lambda_1)}{(\alpha-\lambda_2)(\lambda_2-\beta)}}.
$$
In the new basis, the coordinates of the vector $w_0$ are
$$
\tilde{w}_1=\frac{\alpha-\lambda_2}{\Delta_1 x_P}w_x+\frac{\alpha-\lambda_1}{\Delta_2 x_P}w_y,\quad
\tilde{w}_2=-\frac{\beta-\lambda_2}{\Delta_1 y_P}w_x-\frac{\beta-\lambda_1}{\Delta_2 y_P}w_y.
$$
From Theorem \ref{thm8.4} it follows that the regression line is parallel to $J_{P}^{-1}w_0$. Thus, the vector of the line of regression in the direction $w_0$ is
$
\mathbf{n}_{w}=S^{T}(
\frac{\tilde{w}_1}{J_{P1}}, \frac{\tilde{w}_2}{J_{P2}})^T.
$
 We get that $\mathbf{n}_{w}:=(n_1, n_2)^T$ is given by
$$
\begin{aligned}
n_1&=\frac{1}{(\alpha-\beta)(\lambda_2-\lambda_1)}
\Big[\Big(\frac{(\alpha-\lambda_2)(\beta-\lambda_1)}{J_{P_1}}+\frac{(\alpha-\lambda_1)(\lambda_2-\beta)}{J_{P_2}}\Big)w_x\\
&+\sqrt{(\alpha-\lambda_1)(\alpha-\lambda_2)(\beta-\lambda_1)(\lambda_2-\beta)}\Big(\frac{1}{J_{P_1}}-\frac{1}{J_{P_2}}\Big)w_y\Big]\\
n_2&=\frac{1}{(\alpha-\beta)(\lambda_2-\lambda_1)}
\Big[\sqrt{(\alpha-\lambda_1)(\alpha-\lambda_2)(\beta-\lambda_1)(\lambda_2-\beta)}\Big(\frac{1}{J_{P_1}}
-\frac{1}{J_{P_2}}\Big)w_x\\
&+\Big(\frac{(\alpha-\lambda_1)(\lambda_2-\beta)}{J_{P_1}}+\frac{(\beta-\lambda_1)(\alpha-\lambda_2)}{J_{P_2}}\Big)w_y\Big].
\end{aligned}
$$

 From Example \ref{exam2dim} we know that the formula \eqref{eq:mul} gives the connection of the extremal values  $J_{P_1}$ and $J_{P_2}$ of the principal planar inertia operator at the  point $P$ with the Jacobi elliptic coordinates $\lambda_1$ and $\lambda_2$ of the point $P$, associated with the pencil of confocal conics:
$
J_{P_1}=2J_{1}-m\lambda_1,\quad J_{P_2}=2J_{1}-m\lambda_2.
$
The coordinate transformation between the Cartesian coordinates and the Jacobi elliptic coordinates also gives $x^2_P$ and $y^2_P$
in terms of the Jacobi elliptic coordinates $\lambda_1$ and $\lambda_2$ of the point $P$, associated with the confocal pencil \eqref{eq:pencil2ddirectiona}:
$$
x^2_P=\frac{(\alpha-\lambda_1)(\alpha-\lambda_2)}{\alpha-\beta},\quad y^2_P=\frac{(\beta-\lambda_1)(\beta-\lambda_2)}{\beta-\alpha}.
$$
Thus we get the formula for the line of regression in the direction $w_0$ under the restriction that it contains a given point $P$ in terms
of the Jacobi elliptic coordinates of the point $P$ with respect to the confocal pencil of conics \eqref{eq:pencil2ddirectiona} associated with the given system of points.

The line of the best fit passing through the point $P$ is
$
y-y_P=\kappa(x-x_P), \quad \kappa= -{n_1}/{n_2}.
$
\end{exm}

\begin{exm}\label{ex:pressure} Predicting pressure from the boiling point of water. See \cite{DeSch}, \cite{For}.
The results from experiments
that were trying to obtain a method for estimating altitude were presented in \cite{For}. A formula
is available for altitude in terms of barometric pressure, but it was difficult to carry
a barometer to high altitudes in the XIX century. However, it might be easy for travelers
to carry a thermometer and measure the boiling point of water. Table \ref{tab:water}
contains the measured barometric pressures and boiling points of water from $17$ experiments.
\begin{table}[h!]
  \begin{center}
    \caption{17 Forbes' observations}
    \label{tab:water}
    \begin{tabular}{c|c|c|c|c|c|c}
    \hline
      j&\textbf{1} & \textbf{2} & \textbf{3} & \textbf{4}&\textbf{5} & \textbf{6}\\ 
      \hline
      $x_j$ & 194.5 & 194.3 & 197.9 & 198.4 & 199.4 & 199.9\\
      $y_j$ & 20.79  & 20.79 & 22.40 & 22.67  & 23.15 & 23.35\\
     \hline
    \end{tabular}

  \begin{tabular}{c|c|c|c|c|c|c}
    \hline
      j & \textbf{7} & \textbf{8}&\textbf{9} & \textbf{10}& \textbf{11} & \textbf{12}\\ 
      \hline
      $x_j$ & 200.9 & 201.1 & 201.4 & 201.3 & 203.6 & 204.6 \\ 
      $y_j$ & 23.89  & 23.99 & 24.02 & 24.01 & 25.14 & 26.57  \\ 
     \hline
    \end{tabular}

    \begin{tabular}{c|c|c|c|c|c}
    \hline
      j&\textbf{13} & \textbf{14} & \textbf{15} & \textbf{16}&\textbf{17}\\ 
      \hline
      $x_j$ & 209.5 & 208.6 & 210.7 & 211.9 & 212.2\\ 
      $y_j$ & 28.49  & 27.76 & 29.04 & 29.88  & 30.06 \\ 
     \hline
    \end{tabular}
  \end{center}
\end{table}

One can use the method of least squares to fit a linear relationship between
boiling point and pressure. Let $y_i$ be the pressure for one of Forbes’ observations,
and let $x_i$ be the corresponding boiling point for $i = 1,\dots, 17$. Using the data in
Table \ref{tab:water}, we can compute the covariance matrix. The coordinates of the centroid are $(\bar{x},\bar{y})=(202.9529, 25.05882)$. The components of the hyperplanar inertia operator at the centroid are $J_{xx}=530.78235, J_{yy}=145.93778, J_{xy}=277.54206$. The principal moments of inertia are $J_{1}=0.63839, J_{2}=676.08147$. Then we  compute the line $u_C$ of the best fit as the least-squares line. The intercept and slope of the line $u_C$ are, respectively,
$\hat \beta_0 = -81.06373$ and $\hat\beta_1 = 0.5228$.
 Using formula \eqref{eq:pmidirect}, one calculates the directional moment of inertia $J^w_{u_C}=0.813143014$.

A traveler is
interested in the barometric pressure when the boiling point of water is $201.5$ degrees.
Suppose that this traveler would like to know whether the pressure is $24.5$. Thus, the traveler might wish to test the null hypothesis
$H_0: \beta_0 + 201.5\beta_1 = 24.5$,  versus $H_1:  \beta_0 + 201.5\beta_1 \ne 24.5$.

\begin{figure}[h] \centering
\psset{unit =0.4cm, linewidth=0.5\pslinewidth}
\begin{pspicture}(190.5,19)(228,33)
\psaxes[Ox=195, Oy=20, labelFontSize=\scriptscriptstyle, ticksize=2pt]{->}(195,20)(212,30)
\psdots[dotstyle=Basterisk](194.5,20.79)(194.3,20.79)(197.9,22.4)(198.4,22.67)(199.4,23.15)(199.9, 23.35)(200.9,23.89)(201.1,23.99)(201.4,24.02)(201.3,24.01)(203.6,25.14)(204.6,26.57)(209.5,28.49)(208.6,27.76)(210.7,29.04)(211.9,29.88)(212.2,30.06)
\psdot[dotstyle=*](202.9259,25.05882)
\psdot[dotstyle=*](201.5,24.5)\rput(201.3,24.6){$P$}
\psline(194,20.643983)(212, 29.8984174)
\psparametricplot[plotpoints=200, linecolor=blue, linewidth=0.8pt, fillstyle=none]{-180}{180}{202.9529 3.010289 t cos  mul add 1.370534 t sin mul add 25.05882 5.75007 neg t cos mul add 0.717505 t sin mul add}
\psparametricplot[plotpoints=200, linecolor=red, linewidth=0.8pt, fillstyle=none]{-1}{1}{202.9529 0.08054 t COSH mul add 5.58223 t SINH mul add  25.05882 0.15384 neg t COSH mul add 2.9224495 t SINH mul add}
\psparametricplot[plotpoints=200, linecolor=red, linewidth=0.8pt, fillstyle=none]{-1}{1}{202.9529 0.08054 neg t COSH mul add 5.58223 t SINH mul add  25.05882 0.15384 t COSH mul add 2.9224495 t SINH mul add}
\end{pspicture}
\caption{Example \ref{ex:2ddirectional}: The line of restricted regression and confocal quadrics through the point P.}
\label{fig:restlinreg}
\end{figure}
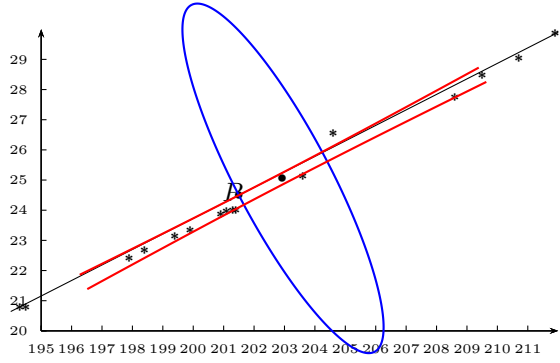

Here we are now interested in the line of regression of $y$ on $x$, i.e. in the vertical direction, under the restriction that it contains the point $P(201.5, 24.5)$.  Now, we introduce $(\tilde{x}, \tilde{y})$ as the principal coordinates at the centroid. Using the coordinate transformation from the original to the principal coordinates, we calculate the coordinates of the point $P$ in the principal system: $\tilde{x}_P=0.08698$, $\tilde{y}_P=-1.68554$.
Similarly, the vertical vector $(0,1)$ in the new coordinate system becomes ${\mathbf{w}}_{0}=(-0.88594, 0.46381)$.
Our pencil of conics associated to the data is defined with $\alpha=0.037552$,
$\beta=-39.69441$:
$$
\frac{\tilde{x}^2}{0.037552-\lambda}+\frac{\tilde{y}^2}{-39.69441-\lambda}=1.
$$
The Jacobi elliptic coordinates of the point $P$ are the solutions of the quadratic equation \eqref{eq:kvadratna}: $\lambda_{1P}=-42.536$, $\lambda_{2P}=0.03049$,. The principal hyperplanar moments of inertia at the point $P$ are
$
J_{1_P}=2J_1-m\lambda_{2_P}=0.75841 \quad J_{2_P}=2J_1-m\lambda_{1_P}=724.3882
$
In the principal coordinates $\mathbf{n}_{w}=\tilde{S}^T(\frac{\tilde{w}_1}{J_{P_1}}, \frac{\tilde{w}_2}{J_{P_2}})^T$.

The equation of the line of the best fit in the sense of least squares passing through the point $P$ is
$
u_P: \tilde{\tilde{y}}=1189.499993\tilde{\tilde{x}}-211.1392.
$
 Using formula \eqref{eq:pmidirect}, we calculate the corresponding directional moment of inertia $J^w_{u_P}=1.455877$.
In the original coordinates, the equation of the line of the best fit $u_P$ is
$
y=0.5141352x-79.0982450.
$

\
The F-statistic
$F=\frac{RSS_1-RSS_2}{p_2-p_1}/\frac{RSS_2}{N-p_2},$
for $N=17$, $p_2=2$, $p_1=1$
is 11.85647, while p-value is $P(F_{1, 15} > 11.85647) = 0.003621119.$
The null hypothesis should be rejected, leading to the conclusion that pressure is not 24.5 (see Fig. \ref{fig:restlinreg}).

\begin{remark}\label{rem:Forbs}
Our calculations from this example do not coincide with the calculations from\- \cite{DeSch}. For example, for unconstrained problem they got
$\hat \beta_0 = -81.049$ while the correct value is $\hat \beta_0 = -81.06373$, as above.
\end{remark}
\

It is interesting to note that in \cite{CR0}, following \cite{Wei}, they discuss this experiment from the point of view of errors in measurement models.

\end{exm}

\subsection*{Acknowledgements}
We are grateful to Maxim Arnold for indicating the reference \cite{Ga} and for interesting discussions. We are also grateful to Pankaj Choudhary, Sam Efromovich, and Frank Konietschke for interesting discussions.  We also express our sincere gratitude to the referees  and editors for their very helpful and constructive remarks and suggestions.

This research has been partially supported by Mathematical Institute of the Serbian Academy of Sciences and Arts, the Science Fund of Serbia grant Integrability and Extremal Problems in Mechanics, Geometry and
Combinatorics, MEGIC, Grant No. 7744592 and the Ministry for Education, Science, and Technological Development of Serbia and the Simons Foundation grant no. 854861.





\end{document}